\documentclass{amsart}

\usepackage{amssymb}
\usepackage{graphicx}
\usepackage{MnSymbol}

\usepackage{enumerate}

\usepackage{amsthm}

\title[double cone harmonic archipelago]{The double cone group is isomorphic to the archipelago group}
\author{Samuel M. Corson}

\theoremstyle{definition}\newtheorem{theorem}{Theorem}
\theoremstyle{definition}
\theoremstyle{definition}

\theoremstyle{definition}\newtheorem{bigtheorem}{Theorem}

\numberwithin{theorem}{section}
\theoremstyle{definition}
\theoremstyle{definition}
\theoremstyle{definition}\newtheorem{definition}[theorem]{Definition}
\theoremstyle{definition}
\theoremstyle{definition}
\theoremstyle{definition}\newtheorem{remark}[theorem]{Remark}
\theoremstyle{definition}
\theoremstyle{definition}\newtheorem{lemma}[theorem]{Lemma}
\theoremstyle{definition}
\theoremstyle{definition}
\theoremstyle{definition}\newtheorem{claim}[theorem]{Claim}
\theoremstyle{definition}

\newcommand{\Red}{\operatorname{Red}}
\newcommand{\GS}{\mathbb{G}\mathbb{S}}
\newcommand{\Aut}{\operatorname{Aut}}
\newcommand{\HA}{\mathbb{H}\mathbb{A}}
\newcommand{\W}{\mathcal{W}}

\newcommand{\pchunk}{\operatorname{p-chunk}}
\newcommand{\Pfine}{\operatorname{Pfine}}
\newcommand{\pindex}{\operatorname{p^*}}
\newcommand{\Close}{\operatorname{Close}}
\newcommand{\coi}{\operatorname{coi}}
\newcommand{\dom}{\operatorname{dom}}
\newcommand{\ran}{\operatorname{ran}}
\newcommand{\Pure}{\operatorname{Pure}}
\newcommand{\se}{\operatorname{se}}
\newcommand{\pa}{\operatorname{par}}

\def\pmc#1{\setbox0=\hbox{#1}
    \kern-.1em\copy0\kern-\wd0
    \kern.1em\copy0\kern-\wd0}

\DeclareMathOperator{\topprod}{\circledast}

\begin{document}

\address{E. T. S. I. I. Universidad Polit\'{e}cnica de Madrid, Jos\'{e} Guti\'{e}rrez Abascal 2, 28006 Madrid, Spain}
\email{sammyc973@gmail.com}

\thanks{The author is partially supported by RYC2023-045493-I.}

\begin{abstract}  We prove the conjecture of James W. Cannon and Gregory R. Conner that the fundamental group of the Griffiths double cone space is isomorphic to that of the harmonic archipelago.  From this and earlier work in this area, we conclude that the isomorphism class of these groups is quite large and includes groups with a great variety of descriptions.

\end{abstract}

\maketitle

\begin{section}{Introduction}

The harmonic archipelago $\HA$ was introduced by Bogley and Sieradski \cite{BS} as a space whose algebraic topology exhibits exotic properties.  The description of $\HA$ is rather innocent.  One takes the closed unit disk in the plane and pushes up countably infinitely many hills of height one, with the bases of the hills shrinking in diameter and converging towards a single point $o$ on the boundary (Figure \ref{harmonicarchipelagofig}).  All of the special behavior of $\HA$ is accumulated at the special point $o$; indeed, $\HA \setminus \{o\}$ is homeomorphic to a disk with a boundary point removed.

The fundamental group $\pi_1(\HA)$ is nontrivial.  For example, the loop which passes clockwise around the ``boundary'' of $\HA$ is nontrivial in $\pi_1(\HA)$ (any purported nulhomotopy would need to pass over infinitely many hills, violating the countinuity of the nulhomotopy).  This may seem initially surprising since any loop in $\HA$ can be homotoped to lie within an arbitrary neighborhood of $o$.  In fact $\pi_1(\HA)$ is uncountable and a combinatorial description of this group involving infinite words is well-known (see Lemma \ref{HAcombinatorial}).  The group $\pi_1(\HA)$ is locally free and every countable locally free group is included as a subgroup (for example the group $\mathbb{Q}$ of rationals or the fundamental group of the complement of the Alexander Horned Sphere) \cite{Hojka}. Any homomorphism from $\pi_1(\HA)$ to $\mathbb{Z}$ is trivial.

Many of the unusual properties of the fundamental group of $\HA$ are shared by that of its longer-studied relative, the Griffiths space $\GS$ (Figure \ref{doubleconefig}).  Introduced by H. B. Griffiths \cite{G}, this space has been a mainstay for wild topologists and used recently in studying notions of infinitary abelianization \cite{BG} and non-abelian cotorsion \cite{EF}.  Combinatorial descriptions of the group $\pi_1(\GS)$ have been of interest in their own right \cite{BZ}.  In common with $\HA$, the space $\GS$ has uncountable fundamental group and has a point arbitrarily close to which any loop can be homotoped.  Perhaps it was these similarities which lead James W. Cannon and Gregory R. Conner to boldly conjecture that $\pi_1(\HA)$ is isomorphic to $\pi_1(\GS)$ \cite{C2}.

For evidence in favor of this conjecture we know that $(\prod_{\omega}\mathbb{Z})/(\sum_{\omega}\mathbb{Z})$ is isomorphic to the abelianization of $\pi_1(\HA)$ \cite[Theorem 1.2]{KR} and is also isomorphic to the abelianization of $\pi_1(\GS)$ \cite[Theorem 1.6]{EF}.  The evident isomorphism between the abelianization of $\pi_1(\HA)$ and that of $\pi_1(\GS)$ is quite abstract and utilizes deep results in abelian group theory and extensive use of the axiom of choice.  Moreover, it is intuitively clear that there is no continuous function from $\HA$ to $\GS$, or vice-versa, which will induce an isomorphism of fundamental groups.  This gives a sense of the difficulty and nonconstructive nature of our proof of the conjecture of Cannon and Conner.

\begin{figure}
\includegraphics[height = 5cm]{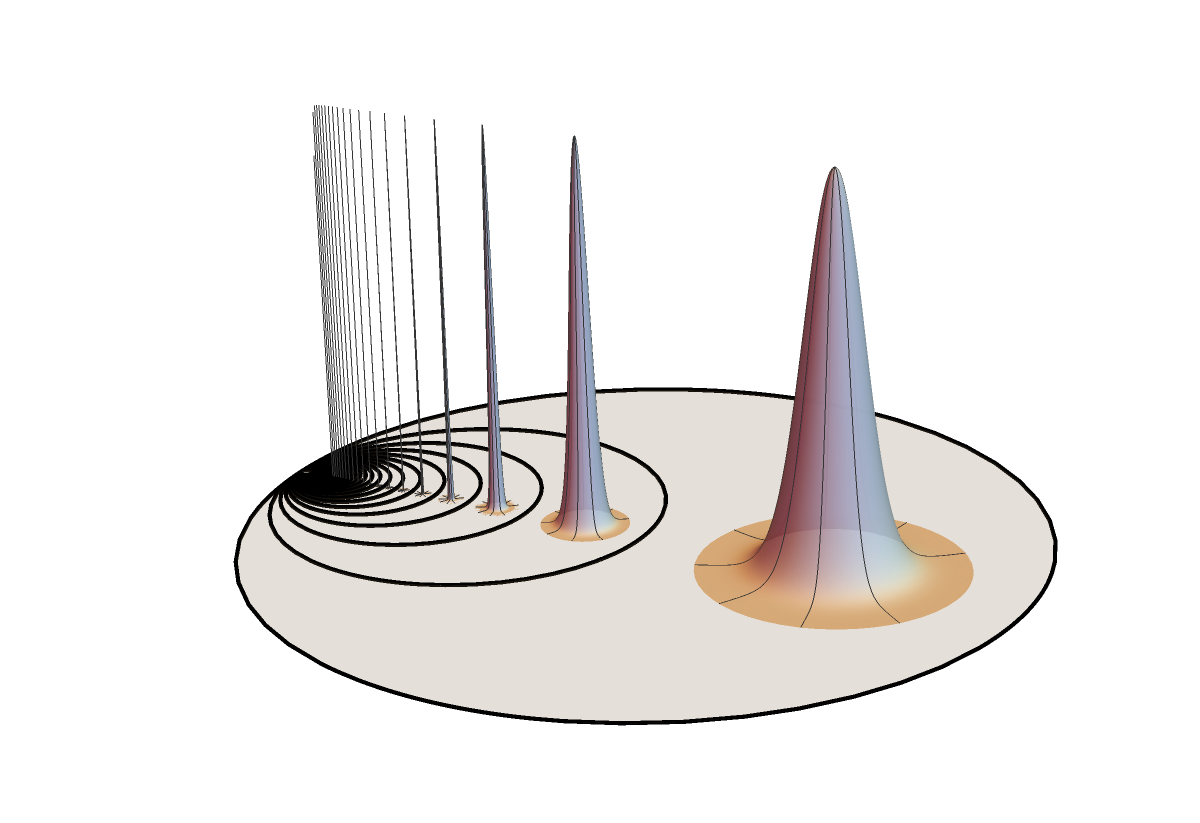}
\caption{The harmonic archipelago $\HA$}
\label{harmonicarchipelagofig}
\end{figure}

\begin{bigtheorem} \label{themain}  The group $\pi_1(\GS)$ is isomorphic to $\pi_1(\HA)$.
\end{bigtheorem}

Theorem \ref{themain} serves as a linchpin allowing us to write a more general statement, towards whose formulation we give some definitions.  Using the notation of \cite{Cor2} we let $\GS_{\kappa}$ denote the metric wedge of $\kappa$-many cones over the infinite earring, where $\kappa$ is a cardinal.  Thus $\GS_1$ is the cone over the infinite earring (hence contractible) and $\GS_2 = \GS$.  When $\kappa$ is in the set $\omega$ of natural numbers, $\GS_{\kappa}$ is a \emph{Peano continuum}  (a nonempty compact metrizable path connected, locally path connected space).

Now for some groups of a different, directly combinatorial, description.  If $\{H_n\}_{n \in \omega}$ is a collection of groups we have an inverse system of homomorphisms $*_{n = 0}^{k + 1} H_n \rightarrow *_{n = 0}^k H_n$ given by deleting the factor $H_{k + 1}$.  Define the topologist product $\topprod_{n \in \omega} H_n$ to be the subset of the inverse limit $\varprojlim *_{n = 0}^k H_n$ consisting of those (often infinite) words $W \in \varprojlim *_{n = 0}^k H_n$ having only finitely many appearances of elements of $H_j$ for each $j \in \omega$.  This subset is in fact a subgroup and satisfies $$*_{n \in \omega} H_n \leq \topprod_{n \in \omega} H_n \leq \varprojlim *_{n = 0}^k H_n.$$  When $H_n \simeq \mathbb{Z}$ for each $n$ the topologist product is isomorphic to the fundamental group of the infinite earring (see Section \ref{Background}).  We define a quotient $\mathcal{A}(\{H_n\}_{n \in \omega})$ by allowing the deletion of finite words, i.e. $$\mathcal{A}(\{H_n\}_{n \in \omega}) = \topprod_{n \in \omega} H_n/\langle\langle *_{n \in \omega} H_n \rangle\rangle.$$  When $H_n \simeq \mathbb{Z}$ for each $n$ the group $\mathcal{A}(\{H_n\}_{n \in \omega})$ is isomorphic to $\pi_1(\HA)$ (via a natural combinatorial function).  The group $\mathcal{A}(\{H_n\}_{n \in \omega})$ is locally free \cite[Theorem 11]{HH}, and assuming $1 < |H_n| \leq 2^{\aleph_0}$ the abelianization is isomorphic to $(\prod_{\omega}\mathbb{Z})/(\sum_{\omega}\mathbb{Z})$  \cite[Theorem 8]{HH}.

Now for the general statement.

\begin{bigtheorem} \label{largeisomorphismclass}
All groups in the following list are in the same isomorphism class.

\begin{enumerate}

\item $\pi_1(\GS_{\kappa})$ where $2 \leq \kappa \leq 2^{\aleph_0}$;

\item $\pi_1(\HA)$;

\item $\mathcal{A}(\{H_n\}_{n \in \omega})$ where $\{H_n\}_{n \in \omega}$ is a sequence of groups without involutions with $1 < |H_n| \leq 2^{\aleph_0}$.

\end{enumerate}
\end{bigtheorem}

That all groups of form (1) are isomorphic was shown by the author in \cite{Cor2}.  The interval $2 \leq \kappa \leq 2^{\aleph_0}$ is optimal since $|\pi(\GS_{\kappa})| = 1$ in case $\kappa = 0, 1$ and $|\pi(\GS_{\kappa})| > 2^{\aleph_0}$ in case $\kappa > 2^{\aleph_0}$.  Kent has shown that if Peano continua $X, Y$ are subspaces of $\mathbb{R}^2$ then $\pi_1(X) \simeq \pi_1(Y)$ if and only if $X$ and $Y$ are homotopy equivalent \cite[Theorem 1.2]{Kent}.  Thus fundamental groups of reasonable spaces in the planar setting are isomorphic precisely if an isomorphism can be realized topologically.  By contrast we know $\pi_1(\GS_2) \simeq \pi_1(\GS_3)$ and no continuous function can induce an isomorphism between the fundamental groups.  Moreover both $\GS_2$ and $\GS_3$ are Peano continua, can be embedded in $\mathbb{R}^3$, and have dimension $2$.  Thus Kent's theorem is fairly sharp.

In \cite{Cor3} the author showed that all groups of form (3) are isomorphic.  An erroneous proof of a less general statement was given by Conner, Hojka and Meilstrup in \cite{CHM}.  As of this writing, it is unknown whether the requirement regarding involutions can be removed.  We have already noted that (2) is isomorphic to a group of form (3), and the isomorphism of (2) with a group of form (1) is Theorem \ref{themain} of the current paper.

From Theorem \ref{largeisomorphismclass} one knows the automorphism group of any group $G$ from that list is enormous: $\Aut(G)$ includes a copy of the group $S_{2^{\aleph_0}}$ of all bijections on a set of size continuum (thus $\Aut(G)$ includes a copy of every group of size at most continuum).  This can be seen by applying \cite[Theorem B]{Cor} or \cite[Corolloary 3.26]{Cor2} to Theorem \ref{largeisomorphismclass}.  Such a $G$ is also locally free \cite[Theorem 11]{HH} and includes all countable locally free groups \cite{Hojka}, and this is not at all obvious for groups of form (1).

\begin{figure}
\includegraphics[height = 5cm]{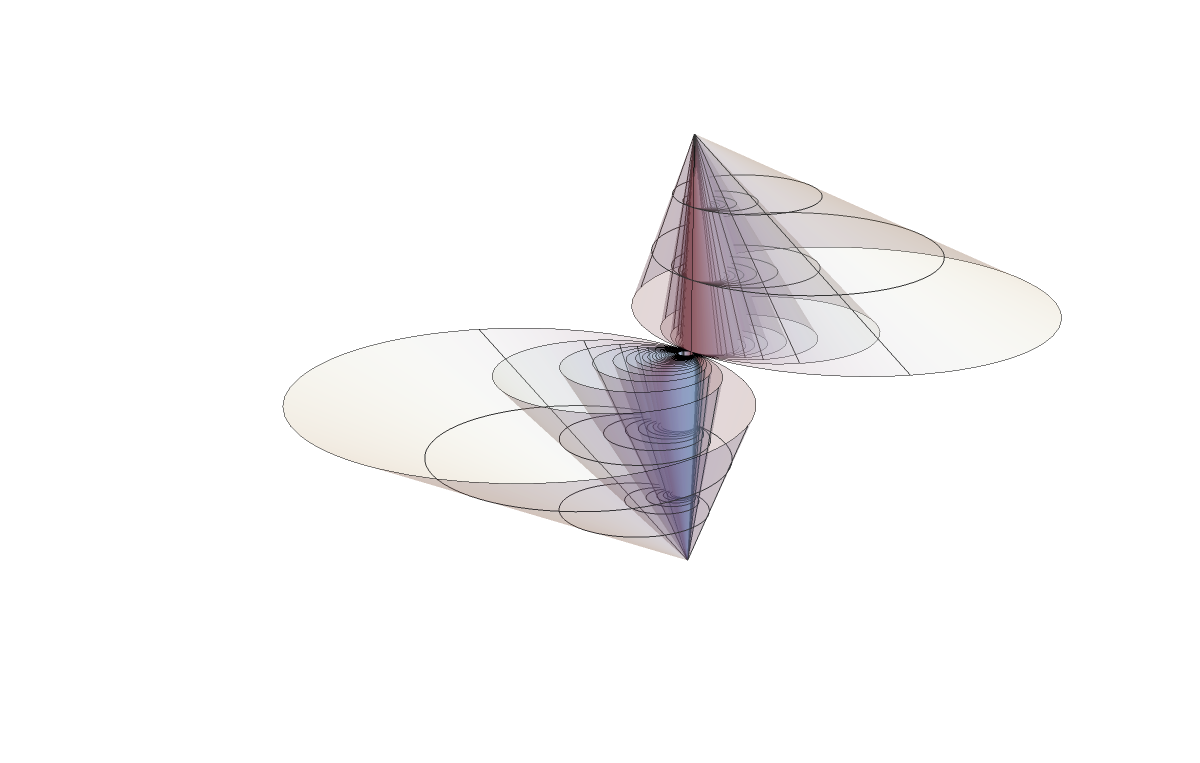}
\caption{The Griffiths double cone $\GS$}
\label{doubleconefig}
\end{figure}

The techniques in the proof Theorem \ref{themain} are modifications of those used by the author in the earlier papers \cite{Cor2}, \cite{Cor3}.  Combinatorics of infinite words, and ways in which to pair up natural segments of words, are described.  Roughly speaking, the ability to pair up more and more such words and to have the pairings tell a ``coherent'' story are what allows the proof to work.  An induction of length continuum renders the desired isomorphism, with many arbitrary choices being made along the way.  The reader may note that groups of form (1) in the statement of Theorem \ref{largeisomorphismclass} seem reasonably similar to one another, and similarly for groups of form (2) and (3).  However form (1) and form (2) have a more distinctive feel.

While this paper does not require familiarity with papers \cite{Cor2} and \cite{Cor3}, some proofs are witheld for the sake of pacing.  This occurs especially in Section \ref{Background}, but a reference to an earlier paper is usually provided.  Section \ref{Background} deals with some details of linear order and word combinatorics, especially certain natural decompositions of words.  Section \ref{strategysection} is devoted to proving the crucial Lemma \ref{homomorphism}.  Section \ref{discrete} deals with extending the collections of pairings when concatenating discretely and Section \ref{nondiscrete} deals with concatenating very nondiscretely.  Section \ref{Concludingarg} contains the final arguments for Theorem \ref{themain}.

\end{section}

\begin{section}{Combinatorial descriptions of the two groups}\label{Background}

In this section we give combinatorial characterizations of $\pi_1(\GS_2)$ and $\pi_1(\HA)$ and present some known facts about so-called reduced words and cancellation schemes.  Important notions regarding purity and decompositions are also provided.  To begin, we describe the fundamental group of the infinite earring (the reader can find more details in \cite{CC}).

\begin{subsection}{Earring group}\label{earringgroup}  The infinite earring $\mathcal{E}$, also known as the Hawaiian earring, is the union $\mathcal{E} = \bigcup_{n \in \omega} C_n$ where $C_n$ is the circle centered at $(\frac{1}{n + 1}, 0) \in \mathbb{R}^2$ of radius $\frac{1}{n + 1}$, with topology inherited from $\mathbb{R}^2$ (see Figure 3).  This space $\mathcal{E}$ is compact, path connected, locally path connected.  The fundamental group of $\mathcal{E}$ has a combinatorial description which resembles that of a free group (although $\pi_1(\mathcal{E})$ is not free \cite{Hig}).

\begin{figure}\label{infiniteearring1}
\includegraphics[width=60mm]{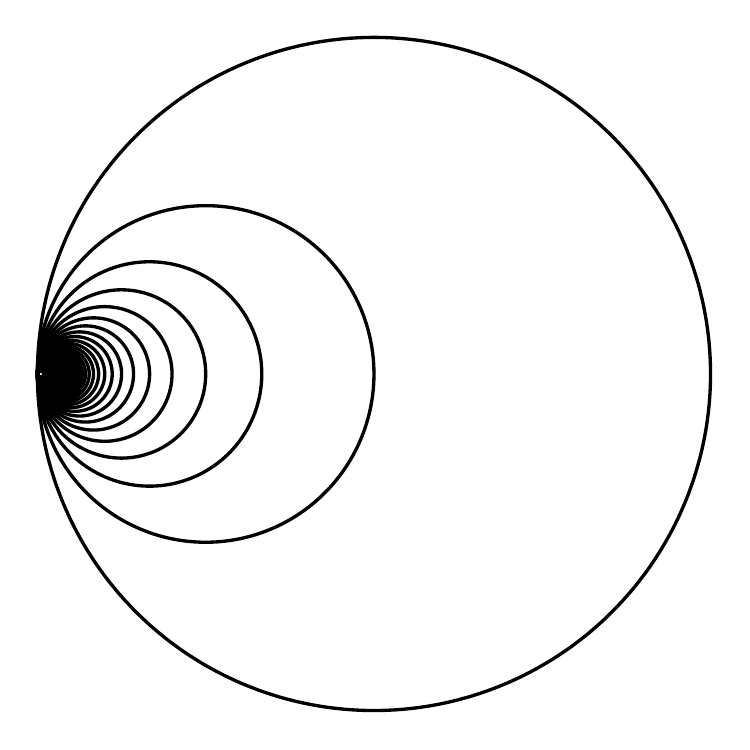}
\caption{The infinite earring $\mathcal{E}$}
\end{figure}

Let $A=\{a_n^{\pm 1}\}_{n\in \omega}$ be a countably infinite set with formal inverses (this is the set of \emph{letters}).  A \emph{word} is a function from a countable totally ordered set $\overline{W}$ to the set $A$ which is finite-to-one.  For example one has the word $W: \omega \rightarrow A$ given by $n \mapsto a_n$, which can ``written'' as $a_0a_1a_2\cdots$.  One can have a word with domain $\mathbb{Q}$ by taking $W: \mathbb{Q} \rightarrow A$ to be an injection.  We consider two words $W$ and $V$ to be the same, and write $W \equiv V$, if there exists an order isomorphism $\iota: \overline{W} \rightarrow \overline{V}$ such that for all $i \in \overline{W}$, $W(i) = V(\iota(i))$.  When such $W$ and $V$ are finite, the $\equiv$ equivalence agrees with the usual syntactic notion that the words read the same way, letter-by-letter.  Let $E$ denote the word with empty domain (the \emph{empty word}).  Let $\W_a$ denote the set of all words, considered up to $\equiv$, in the alphabet $A$.

Each word $W$ has an inverse, denoted $W^{-1}$, given by taking $\overline{W^{-1}}$ to be the set $\overline{W}$ but having the reverse order and $W^{-1}(i) = (W(i))^{-1}$.  Given two words $W_0$ and $W_1$ we form their concatenation $W_0W_1$ by taking the domain $\overline{W_0W_1}$ to be the disjoint union $\overline{W_0} \sqcup \overline{W_1}$ which is ordered to extend the orders of $\overline{W_0}$ and $\overline{W_1}$ and places elements of $\overline{W_0}$ below those of $\overline{W_1}$.  The function $W_0W_1$ is given by

\[
W_0W_1(i) = \left\{
\begin{array}{ll}
W_0(i)
                                            & \text{if } i \in \overline{W_0}, \\
W_1(i)                                        & \text{if } i \in \overline{W_1}.
\end{array}
\right.
\]

\noindent Concatenation is evidently an associative operation.  We say $V$ is a \emph{subword} of $W$ if we can write $W \equiv W_0VW_1$ for some words $W_0$ and $W_1$.  If $W_0 \equiv E$ in such a writing then we say $V$ is an \emph{initial} subword of $W$; if $W_1 \equiv E$ then we say $V$ is a \emph{terminal} subword.  We will frequently utilize a notion of \emph{infinite concatenation}.  If $W$ is nonempty let $d(W) = k$ where $k$ is the minimal subscript among the letters in the image of $W$; and $d(E) = \infty$.  Suppose that $\{W_{\lambda}\}_{\lambda \in \Lambda}$ is a collection of words indexed by a totally ordered set $\Lambda$ and that for every $N > 0$ the set $\{\lambda \in \Lambda \mid d(W_{\lambda}) \leq N \}$ is finite.  The concatenation $W \equiv \prod_{\lambda \in \Lambda} W_{\lambda}$ has domain $\bigsqcup_{\lambda \in \Lambda} \overline{W_{\lambda}}$ ordered in the natural way and we let $W(i) = W_{\lambda}(i)$ where $i \in \overline{W_{\lambda}}$.  It is clear that this function is a word since every appearance of the letter $a_k^{-1}$ will be in a $W_{\lambda}$ with $d(W_{\lambda}) \leq k$.  If each word $W_{\lambda}$ is nonempty we know the index $\Lambda$ is countable.

For $n \in \omega$ and word $W$ let $p_n(W)$ be the finite word given by the restriction $W\upharpoonright \{i \in \overline{W} \mid W(i) \in \{a_m^{\pm 1}\}_{0 \leq m \leq n}\}$.  Let $W \sim V$ if for all $n \in \omega$ the words $p_n(W)$ and $p_n(V)$ are equal as words in the free group $F(a_0, \ldots, a_{n})$.  Letting $[W]$ be the $\sim$ equivalence class of the word $W$, we obtain a group operation on $\W_a/\sim$ by defining $[W_0][W_1] := [W_0W_1]$.  The identity in the group is the equivalence class $[E]$ and the inverse is reasonably $[W]^{-1} = [W^{-1}]$.  This group is isomorphic to $\pi_1(\mathcal{E})$ and has cardinality $2^{\aleph_0}$.

As in a free group, we would prefer that the group elements be words instead of equivalence classes of words.  We say a word $W$ is \emph{reduced} if whenever we write a concatenation $W \equiv W_0W_1W_2$ with $W_1 \sim E$ we have $W_1 \equiv E$.  This definition is clearly an extension of that in free groups, and many of the same results hold.  For example we have the following (see \cite[Theorem 1.4, Corollary 1.7]{E}).

\begin{lemma}\label{reduced}  For each $W \in \W_a$ there exists a unique, up to $\equiv$, reduced word $W_0 \in  [W]$.  Furthermore if $W$ and $U$ are reduced there exist unique words $W_0, W_1, U_0, U_1$ such that 

\begin{enumerate}  \item $W \equiv W_0W_1$;

\item $U \equiv U_0U_1$;

\item $W_1\equiv U_0^{-1}$;

\item $W_0U_1$ is reduced.
\end{enumerate}

\end{lemma}

Let $\Red_a$ denote the set of reduced words in $\W_a$ and for $W \in \W_a$ let $\Red(W)$ be the (unique up to $\equiv$) reduced word such that $W \sim \Red(W)$.    The following is immediate.

\begin{lemma}\label{reducedconsequence}  If $W, U \in \W_a$ we have $\Red(WU) \equiv \Red(\Red(W)\Red(U))$.  Similarly, given $W_0, W_1, W_2 \in \W_{\kappa}$ we have $\Red(W_0W_1W_2) \equiv \Red(W_0\Red(W_1W_2)) \equiv \Red(\Red(W_0W_1)W_2)$.
\end{lemma}

From Lemmas \ref{reduced} and \ref{reducedconsequence} the group $\W_a/\sim$ is isomorphic to the set $\Red_a$ under the group operation $W*U =\Red(WU)$.  We give the following definition (see \cite[Definition 3.4]{CC}):

\begin{definition}\label{cancellation}  Given $W\in \W_a$ we say $\mathcal{S} \subseteq \overline{W} \times\overline{W}$ is a \emph{cancellation scheme} on $W$ provided

\begin{enumerate}\item for $\langle i_0, i_1 \rangle \in \mathcal{S}$ we have $i_0 < i_1$;

\item  if $\langle i_0, i_1\rangle \in \mathcal{S}$ and $\langle i_0, i_2\rangle \in \mathcal{S}$ then $i_2 = i_1$;

\item  if $\langle i_0, i_1\rangle \in \mathcal{S}$ and $\langle i_2, i_1\rangle\in \mathcal{S}$ then $i_2 = i_0$;

\item  if $\langle i_0, i_1 \rangle \in \mathcal{S}$ and $i_2\in (i_0, i_1) \subseteq \overline{W}$ there exists $i_3\in (i_0, i_1)$ such that either $\langle i_2, i_3\rangle \in \mathcal{S}$ or $\langle i_3, i_2\rangle \in \mathcal{S}$;

\item  if $\langle i_0, i_1\rangle \in \mathcal{S}$ then $W(i_0) = (W(i_1))^{-1}$.

\end{enumerate}

\end{definition}

We are using $\langle \cdot, \cdot \rangle$ to denote an ordered pair and $(*, *)$ to denote an open interval.  We'll also use $\langle \cdot \rangle$ to denote a generated subgroup (a lack of a comma makes this unambiguous).

By Zorn's Lemma, every cancellation scheme $\mathcal{S}$ on a word $W$ is included in a maximal cancellation scheme $\mathcal{S}'$, i.e. $\mathcal{S} \subseteq \mathcal{S}'$ and $\mathcal{S}'$ is not a proper subset of a cancellation scheme on $W$.  A maximal cancellation scheme discloses the reduced word representative, just the same as freely reducing a finite word until free reductions are no longer possible.  The following is \cite[Theorem 3.9]{CC}:

\begin{lemma}\label{cancellationreduces}  If $\mathcal{S}$ is a maximal cancellation for $W\in \W_a$ then the restriction

\begin{center}
$W\upharpoonright \{i\in \overline{W}\mid (\neg \exists i')(\langle i, i'\rangle \in \mathcal{S} \text{ or }\langle i, i'\rangle \in \mathcal{S})\} \equiv\Red(W)$.
\end{center}

\noindent More particularly a word has only trivial cancellation scheme if and only if that word is reduced.  Thus if $W \in \W_a$ with $W \equiv \prod_{\lambda \in \Lambda} W_{\lambda}$ then $\Red(W) \equiv \Red(\prod_{\lambda \in \Lambda}\Red(W_{\lambda}))$.
\end{lemma}

Considering, for example, the word

\begin{center}

$W \equiv a_0a_1a_2\cdots \cdots a_2^{-1}a_1^{-1}a_0^{-1}$

\end{center}

\noindent we know that this word is $\sim$ to the empty word $E$.  A (maximal) cancellation scheme is given by $\{\langle \min\overline{W}, \max\overline{W}\rangle, \langle \min(\overline{W} \setminus \min\overline{W}), \max(\overline{W} \setminus \max\overline{W})\rangle, \ldots\}$.
\end{subsection}

\begin{subsection}{The Griffiths space and harmonic archipelago groups}

Having reviewed a combinatorial description of the earring group, we now give the combinatorial description of the fundamental group of the Griffiths space and then of the harmonic archipelago.  Beginning with the Griffiths space, we take a larger countably infinite set with formal inverses $AB = \{a_n^{\pm 1}\}_{n \in \omega} \sqcup \{b_n^{\pm 1}\}_{n \in \omega}$.  The ideas and definitions  from subsection \ref{earringgroup} are applied to this setting as well.  For example a word is a function $W: \overline{W} \rightarrow AB$ which is finite-to-one, the concept $\equiv$ is defined in the same way.  The set of words in $AB$, up to $\equiv$, will be denoted $\W_{a, b}$.  For each $n \in \omega$, $p_n(W) = W \upharpoonright \{i \in \overline{W} \mid W(i) \in \{a_m^{\pm 1}\}_{0 \leq m \leq n} \cup \{b_m^{\pm 1}\}_{0 \leq m \leq n}\}$.  For words $W, V \in \W_{a, b}$ we write $W \sim V$ under precisely the same circumstances as before, reduced words are defined in the same way, and Lemmas \ref{reduced}, \ref{reducedconsequence} and \ref{cancellationreduces} hold.  Write $\Red_{a, b}$ for the group of reduced words in the alphabet $AB$.  Incidentally, it is easy to see that $\Red_a$ is isomorphic to $\Red_{a, b}$, by extending the map on letters $A \rightarrow AB$ given by $a_{2n} \mapsto a_{n}$, $a_{2n + 1} \mapsto b_n$.

We will say a word $W \in \Red_{a, b}$ is \emph{$a$-pure} (respectively \emph{$b$-pure}) if all the letters which appear in $W$ are in $\{a_n^{\pm 1}\}_{n \in \omega}$ (respectively $\{b_n^{\pm 1}\}_{n \in \omega}$).  A word $W \in \Red_{a, b}$ is \emph{pure} provided it is $a$-pure or it is $b$-pure.  Write $\Pure_{a, b}$ for the set of pure words in $\Red_{a, b}$.  The following is well-known (see, for example, \cite[Theorem 2.8]{Cor2}).

\begin{lemma}\label{GScombinatorial}  The fundamental group $\pi_1(\GS_2)$ is isomorphic to the quotient $$\Red_{a, b}/\langle\langle \Pure_{a, b}\rangle\rangle.$$
\end{lemma}

Now we take a different alphabet $C = \{c_n^{\pm 1}\}_{n \in \omega}$ and let $\W_c$ denote the set of words using the alphabet $C$, up to $\equiv$, and $\Red_c$ denote the set of reduced words.  If $n \in \omega$ we shall say that $U \in \Red_c$ is $n$-pure if $n$ is the only subscript of the letters appearing in $U$, which is evidently the same thing as being a word of the form $c_n^z$ where $z \in \mathbb{Z}$.  The empty word $E$ is $n$-pure for each $n \in \omega$.  We'll say a word $U \in \Red_c$ is \emph{pure} provided it is $n$-pure for some $n \in \omega$.  Let $\Pure_c$ denote the set of pure words in $\Red_c$.  The following is a consequence of \cite[Theorem 5]{CHM}.

\begin{lemma}\label{HAcombinatorial}  The fundamental group $\pi_1(\HA)$ is isomorphic to the quotient $$\Red_c/\langle\langle \Pure_c\rangle\rangle.$$
\end{lemma}

For a word $W \in \Red_{a, b}$ we will let $[[W]]$ denote the equivalence class of $W$ in the quotient $\Red_{a, b}/\langle\langle \Pure_{a, b}\rangle\rangle$.  We will similarly let $[[U]]$ denote the equivalence class of a word $U \in \Red_c$ in the quotient $\Red_c/\langle\langle \Pure_c\rangle\rangle$.  The usage is unambiguous since the alphabets $AB$ and $C$ are disjoint.  Let $\beth_{a, b}$ denote the quotient homomorphism from $\Red_{a, b}$ to $\Red_{a, b}/\langle\langle \Pure_{a, b}\rangle\rangle$, and similarly $\beth_c$ that from $\Red_c$ to $\Red_c/\langle\langle \Pure_c\rangle\rangle$.

\end{subsection}

\begin{subsection}{Pure decomposition}

In this subsection we shall review a very natural way in which to decompose a word in $\Red_{a, b}$, or in $\Red_c$, as a concatenation of pure subwords.  For this we introduce an abuse of notation as follows.  If $\{\Lambda_{\lambda}\}_{\lambda \in \Lambda}$ is a collection of totally ordered sets, with $\Lambda$ also totally ordered, then the concatenation $\prod_{\lambda \in \Lambda} \Lambda_{\lambda}$ is the totally ordered set whose underlying set is the disjoint union $\bigsqcup_{\lambda \in \Lambda} \Lambda_{\lambda}$.  The order on $\prod_{\lambda \in \Lambda} \Lambda_{\lambda}$ is the one which extends the order on each of the $\Lambda_{\lambda}$ and places elements of $\Lambda_{\lambda}$ below those in $\Lambda_{\lambda'}$ if $\lambda$ is below $\lambda'$ in $\Lambda$.  For totally ordered sets we also write $\Lambda' = \prod_{\lambda \in \Lambda} I_{\lambda}$ if $\Lambda'$ is the disjoint union of intervals $I_{\lambda} \subseteq \Lambda'$ and the elements of $I_{\lambda}$ are all below those of $I_{\lambda'}$ when $\lambda$ is below $\lambda'$ in $\Lambda$.  By interval, we mean a convex subset in a totally ordered set, so an interval may be empty, may contain a minimal element, etc.  An interval $I \subseteq \Lambda$ is \emph{initial} if $\Lambda \setminus I$ contains no elements below an element in $I$, and the definition of a \emph{terminal} interval is dual.

Given a word $W \in \Red_{a, b}$ we decompose the domain $\overline{W} = \prod_{\lambda \in \Lambda} I_{\lambda}$ such that each $I_{\lambda}$ is a nonempty maximal interval with $W \upharpoonright I_{\lambda}$ pure.  We will call this the \emph{pure decomposition of the domain of} $W$ (abbreviated $p$-decomposition of the domain of $W$).  Notice that the index $\Lambda$, which we will call the \emph{$p$-index} and often denote $p^*(W)$, is defined up to order isomorphism.  Write $W \equiv_p \prod_{\lambda \in \Lambda} W_{\lambda}$ to express that $\overline{W} = \prod_{\lambda \in \Lambda} \overline{W_{\lambda}}$ is the $p$-decomposition of the domain of $W$ (this will be called the $p$-decomposition of $W$).  The $p$-index of the empty word $E$ will be the empty set.

If $W \equiv_p \prod_{\lambda \in p^*(W)} W_{\lambda}$ and $I$ is an interval in $p^*(W)$ then we let $W \upharpoonright_p I$ denote the word $\prod_{\lambda \in I} W_{\lambda}$.  Given words $W, W' \in \Red_{a, b}$ we say that $W'$ is a \emph{p-chunk} of $W$ if for some interval $I \subseteq p^*(W)$ we have $W' \equiv W \upharpoonright_p I$.  For $W \in \Red_{a, b}$ we let $\pchunk(W)$ denote the set of p-chunks of $W$.  If a p-chunk of a word $W \equiv_p \prod_{\lambda \in p^*(W)} W_{\lambda}$ is pure, then evidently it is the empty word or it is one of the $W_{\lambda}$.

The pure decomposition behaves in predictable ways when one concatenates and reduces two elements in $\Red_{a, b}$.  We recount some lemmas which will be used.

\begin{lemma}\label{pchunkmultiplication}\cite[Lemma 3.1]{Cor2}  Suppose that $W, W' \in \Red_{a, b}$ with $W \equiv_p \prod_{\lambda \in \Lambda} W_{\lambda}$ and $W' \equiv_p \prod_{\lambda' \in \Lambda'} W_{\lambda'}'$.  Then there exists a (possibly empty) initial interval $I \subseteq \Lambda$, a (possibly empty) terminal interval $I' \subseteq \Lambda'$ such that either

\begin{enumerate}[(i)] \item $\Red(WW') \equiv_p \prod_{\lambda \in I}W_{\lambda}\prod_{\lambda'\in I'}W_{\lambda'}'$; or

\item there exist $\lambda_0\in \Lambda$ which is the least element strictly above all elements in $I$, $\lambda_1\in \Lambda'$ which is the greatest element strictly below all elements of $I'$ and 

\begin{center}
$\Red(WW') \equiv_p (\prod_{\lambda \in I}W_{\lambda})V(\prod_{\lambda'\in I'}W_{\lambda'}')$ 
\end{center}

\noindent where $V \equiv \Red(W_{\lambda_0}W_{\lambda_1}')\not\equiv E$ is pure.
\end{enumerate}

\end{lemma}

\begin{lemma}\label{elementsofthegeneratedsubgroup}\cite[Lemma 3.2]{Cor2}  Suppose that $X \subseteq \Red_{a, b}$.  For each nonempty element $V$ of the generated subgroup $\langle \bigcup_{W \in X}\pchunk(W) \rangle \leq \Red_{a, b}$ if $V \equiv_p \prod_{\lambda \in \Lambda} W_{\lambda}$ then there exist nonempty intervals $I_0, \ldots, I_n$ in $\Lambda$ such that

\begin{enumerate}[(i)]

\item $\Lambda = \prod_{i = 0}^n I_i$; and

\item for each  $0 \leq i \leq n$ at least one of the following holds:

\begin{enumerate}[(a)]

\item $I_i$ is a singleton $\{\lambda\}$ such that $W_{\lambda}$ is the reduction of a finite concatenation of pure p-chunks of elements in $X^{\pm 1}$;

\item $\prod_{\lambda \in I_i} W_{\lambda}$ is a p-chunk of some element in $X^{\pm 1}$.

\end{enumerate}

\end{enumerate}
\end{lemma}

A subgroup $G \leq \Red_{a, b}$ is \emph{p-fine} if each p-chunk $V$ of each $W \in G$ is also an element of $G$.

\begin{lemma}\label{fine}\cite[Lemma 3.3]{Cor2}  If $X \subseteq \Red_{a, b}$ then the subgroup $\langle \bigcup_{W \in X} \pchunk(W)\rangle \leq \Red_{a, b}$ is p-fine.  This is the smallest p-fine subgroup including the set $X$.
\end{lemma}

For a subset $X \subseteq \Red_{a, b}$ we let $\Pfine(X)$ denote the smallest p-fine subgroup of $\Red_{a, b}$ which includes $X$.

We conclude this subsection by defining the analogous decomposition for words in $\Red_c$ and stating the accompanying lemmas.  For a word $U \in \Red_c$ we decompose $\overline{U} = \prod_{\theta \in \Theta} I _{\theta}$ where each $I_{\theta}$ is a maximal nonempty interval in $\overline{U}$ such that $U \upharpoonright I_{\theta}$ is pure.  This decomposition is clearly unique and we call it the \emph{pure decomposition of the domain of $U$}.  The accompanying decomposition $U \equiv \prod_{\theta \in \Theta} U_{\theta}$ is the $p$-decomposition of $U$, we will generally write $p^*(U)$ for the (uniquely defined up to order isomorphism) index $\Theta$, and write $U \equiv_p \prod_{\theta \in p^*(U)} U_{\theta}$ to say that the expressed decomposition is the $p$-decomposition.  Define $U\upharpoonright_p I$ for an interval $I \subseteq p^*(U)$ in the same way as before.  For $U, U' \in \Red_c$ we'll say $U'$ is a p-chunk of $U$ if for some interval $I \subseteq p^*(U)$ we have $U' \equiv U \upharpoonright_p I$.  Clearly if $U \equiv_p \prod_{\theta \in p^*(U)} U_{\theta}$ then each $U_{\theta}$ is finite.

Write $\pchunk(U)$ for the set of all p-chunks of $U$. A subgroup $G \leq \Red_c$ is \emph{p-fine} if each p-chunk of an element of $G$ is again an element of $G$.  The following lemmas follow in the same way as their analogues above.

\begin{lemma}\label{pchunkmultiplicationC}  Suppose that $U, U' \in \Red_{c}$ with $U \equiv_p \prod_{\theta \in \Theta} U_{\theta}$ and $U' \equiv_p \prod_{\theta' \in \Theta'} U_{\theta'}'$.  Then there exists a (possibly empty) initial interval $I \subseteq \Theta$, a (possibly empty) terminal interval $I' \subseteq \Theta'$ such that either

\begin{enumerate}[(i)] \item $\Red(UU') \equiv_p \prod_{\theta \in I}U_{\theta}\prod_{\theta'\in I'}U_{\theta'}'$; or

\item there exist $\theta_0\in \Theta$ which is the least element strictly above all elements in $I$, $\theta_1\in \Theta'$ which is the greatest element strictly below all elements of $I'$ and 

\begin{center}
$\Red(UU') \equiv_p (\prod_{\theta \in I}U_{\theta})V(\prod_{\theta'\in I'}U_{\theta'}')$ 
\end{center}

\noindent where $V \equiv \Red(U_{\theta_0}U_{\theta_1}')\not\equiv E$ is pure.
\end{enumerate}

\end{lemma}

\begin{lemma}\label{elementsofthegeneratedsubgroupC}  Suppose that $X \subseteq \Red_{c}$.  For each nonempty element $V$ of the generated subgroup $\langle \bigcup_{U \in X}\pchunk(U) \rangle \leq \Red_{c}$ if $V \equiv_p \prod_{\theta \in \Theta} U_{\theta}$ then there exist nonempty intervals $I_0, \ldots, I_n$ in $\Theta$ such that

\begin{enumerate}[(i)]

\item $\Theta = \prod_{i = 0}^n I_i$; and

\item for each  $0 \leq i \leq n$ at least one of the following holds:

\begin{enumerate}[(a)]

\item $I_i$ is a singleton $\{\theta\}$ such that $U_{\theta}$ is the reduction of a finite concatenation of pure p-chunks of elements in $X^{\pm 1}$;

\item $\prod_{\theta \in I_i} U_{\theta}$ is a p-chunk of some element in $X^{\pm 1}$.

\end{enumerate}

\end{enumerate}
\end{lemma}

A subgroup $G \leq \Red_{c}$ is \emph{p-fine} if each p-chunk $V$ of each $U \in G$ is also an element of $G$.

\begin{lemma}\label{fineC}  If $X \subseteq \Red_{c}$ then the subgroup $\langle \bigcup_{U \in X} \pchunk(V)\rangle \leq \Red_{c}$ is p-fine.  This is the smallest p-fine subgroup including the set $X$.
\end{lemma}

Write $\Pfine(X)$ for the smallest p-fine subgroup of $\Red_c$ which includes the set $X \subseteq \Red_c$.

\end{subsection}

\begin{subsection}{Close subsets}\label{closesubsets}

Now we recount a useful concept which will be applied throughout the rest of the paper.

\begin{definition}  Given a totally ordered set $\Lambda$, we say $\Lambda_0 \subseteq \Lambda$ is \emph{close in} $\Lambda$, denoted $\Close(\Lambda_0, \Lambda)$, provided every infinite interval in $\Lambda$ has nonempty intersection with $\Lambda_0$. 

\end{definition}

For example, a subset of the set $\omega$ of natural numbers is close in $\omega$ if and only if it is infinite.  A subset of $\mathbb{Q}$ is close if and only if it is dense.

\begin{lemma}\label{basiccloseproperties}\cite[Lemma 3.6]{Cor2}  The following hold:

\begin{enumerate}[(i)]

\item  If $\Close(\Lambda_0, \Lambda)$ then for any infinite interval $I \subseteq \Lambda$ the set $I \cap \Lambda_0$ is infinite.

\item  If $\Lambda_2 \subseteq \Lambda_1 \subseteq \Lambda_0$ with $\Close(\Lambda_{i+1}, \Lambda_i)$  for $i = 0, 1$, then $\Close(\Lambda_2, \Lambda_0)$.

\item  If $\Lambda \equiv \prod_{\theta \in \Theta} \Lambda_{\theta}$, $\Close(\Lambda_{\theta, 0}, \Lambda_{\theta})$ for each $\theta \in \Theta$, and $\Close(\{\theta \in \Theta \mid \Lambda_{\theta, 0} \neq \emptyset\}, \Theta)$ then $\Close(\bigcup_{\theta \in \Theta} \Lambda_{\theta, 0}, \Lambda)$.

\item  If  $I_0$ is an interval in $\Lambda$ and $\Close(\Lambda_0, \Lambda)$ then $\Close(\Lambda_0 \cap I_0, I_0)$

\end{enumerate}

\end{lemma}

We give some more notation.  If $\Close(\Lambda_0, \Lambda)$ then for an interval $I \subseteq \Lambda$ we let $\varpropto(I, \Lambda_0)$ denote the smallest interval in $\Lambda$ which includes the set $I \cap \Lambda_0$.  In other words $$\varpropto(I, \Lambda_0) =  \bigcup\{[\lambda_0, \lambda_1] \mid \lambda_0, \lambda_1 \in I \cap \Lambda_0\}.$$

\begin{lemma}\label{prettyclose}\cite[Lemma 3.7]{Cor2} Let $\Close(\Lambda_0, \Lambda)$ and $I \subseteq \Lambda$ be an interval.

\begin{enumerate}[(i)]

\item The inclusion $I \supseteq \varpropto(I, \Lambda_0)$ holds and $\varpropto(I, \Lambda_0) = \varpropto(\varpropto(I, \Lambda_0), \Lambda_0)$.

\item  The set $I \setminus \varpropto(I, \Lambda_0)$ is the disjoint union of an initial and terminal subinterval $I_0, I_1 \subseteq I$ (either subinterval could be empty) with $|I_0|, |I_1| < \infty$.
\end{enumerate}

\end{lemma}

We'll say totally ordered sets $\Lambda$ and $\Theta$ are \emph{close-isomorphic} if there exist $\Lambda_0 \subseteq \Lambda$ and $\Theta_0 \subseteq \Theta$ with $\Close(\Lambda_0, \Lambda)$, $\Close(\Theta_0, \Theta)$ and $\Lambda_0$ order isomorphic to $\Theta_0$; and if $\iota$ is an order isomorphism between such a $\Lambda_0$ and $\Theta_0$ then we will call $\iota$ a \emph{close order isomorphism from $\Lambda$ to $\Theta$}.  Evidently the inverse of a close order isomorphism from $\Lambda$ to $\Theta$ is a close order isomorphism from $\Theta$ to $\Lambda$.  We'll abbreviate close order isomorphism with the expression \emph{coi}.

Given coi $\iota$ between $\Lambda$ and $\Theta$, with $\iota: \Lambda_0 \rightarrow \Theta_0$, and an interval $I \subseteq \Lambda$ we let $\varpropto(I, \iota)$ denote the smallest interval in $\Theta$ which includes the set $\iota(I \cap \Lambda_0)$.  In other words, $\varpropto(I, \iota) = \bigcup \{[\theta_0, \theta_1]\mid \theta_0, \theta_1 \in \iota(I \cap \Lambda_0)\}$, where each interval $[\theta_0, \theta_1]$ is being considered in $\Theta$.

\begin{lemma}\label{almostidentified}\cite[Lemma 3.8]{Cor2}  If $\iota: \Lambda_0 \rightarrow \Theta_0$ is a coi between $\Lambda$ and $\Theta$ and $I \subseteq \Lambda$ is an interval then $\varpropto(\varpropto(I, \iota), \iota^{-1}) = \varpropto(I, \Lambda_0)$.
\end{lemma}

Note that a coi $\iota$ between $\Lambda$ and $\Theta$ also induces a coi between the reversed orders $\Lambda^{-1}$ and $\Theta^{-1}$ in the obvious way.

\begin{lemma}\label{coilemma}\cite[Lemma 3.9]{Cor2}  Let $\Lambda \equiv I_0 \cdots I_n$ and $\iota: \Lambda_0 \rightarrow \Theta_0$ a coi from $\Lambda$ to $\Theta$.  Then there exist (possibly empty) finite subintervals $I_0', \ldots, I_{n+1}'$ of $\varpropto(I, \iota)$ such that 

\begin{center}

$\varpropto(\Lambda, \iota) \equiv I_0'\varpropto(I_0, \iota) I_1' \varpropto(I_1, \iota) I_2' \cdots \varpropto(I_n, \iota) I_{n+1}'$.

\end{center}
\end{lemma}

\begin{lemma}\label{coilemma2}\cite[Lemma 3.10]{Cor2}  Let $\iota: \Lambda_0 \rightarrow \Theta_0$ be a coi from $\Lambda$ to $\Theta$.  If $I \subseteq \Lambda$ is finite then $\varpropto(I, \iota)$ is finite.
\end{lemma}

\end{subsection}

\end{section}

\begin{section}{A strategy for constructing an isomorphism}\label{strategysection}

In this section we introduce the fundamental building blocks used to construct the desired isomorphism.  If $W \in \Red_{a, b}$ and $U \in \Red_c$ we write $\coi(W, \iota, U)$ to say that $\iota$ is a coi between $p^*(W)$ and $p^*(U)$ (by abuse of language we'll sometimes say that $\iota$ is a coi from $W$ to $U$).  Now we give a long definition.

\begin{definition}  A collection $\{\coi(W_x, \iota_x, U_x)\}_{x\in X}$ of coi triples is \emph{coherent} if for any choice of $x_0, x_1 \in X$, intervals $I_0 \subseteq \pindex(W_{x_0})$ and $I_1 \subseteq \pindex(W_{x_1})$ and $i \in \{-1, 1\}$ such that $W_{x_0} \upharpoonright_p I_0 \equiv (W_{x_1}\upharpoonright_p I_1)^i$ we get

\begin{center}

$[[U_{x_0}\upharpoonright_p \varpropto(I_0, \iota_{x_0})]] = [[(U_{x_1}\upharpoonright_p \varpropto(I_1, \iota_{x_1}))^i]]$

\end{center}
 
\noindent and similarly for any choice of $x_2, x_3 \in X$, intervals $I_2 \subseteq \pchunk(U_{x_2})$ and $I_3 \subseteq \pchunk(U_{x_3})$ and $j \in \{-1, 1\}$ such that $U_{x_2} \upharpoonright_p I_2 \equiv (U_{x_3} \upharpoonright_p I_3)^j$ we get

\begin{center}

$[[W_{x_2} \upharpoonright_p \varpropto(I_2, \iota_{x_2}^{-1})]] = [[(W_{x_3} \upharpoonright_p \varpropto(I_3, \iota_{x_3}^{-1}))^j]]$.

\end{center}

\end{definition}

\noindent The symmetric nature of this definition will be key in producing the isomorphism.  Note that a word can appear multiple times in a coherent collection.  For example, if each element of $\{W_x\}_{x \in X}$ is pure then the collection $\{(W_x, \iota_x, E)\}_{x \in X}$ is obviously coherent (taking each $\iota_x$ to be the empty function).  We shall also see that it is rather tedious to check that a collection of coi is coherent.  As an illustration, if the collection $\{\coi(W, \iota, U)\}$ has only one element, and $p^*(W)$ is order isomorphic to $\mathbb{Q}$, then one must consider uncountably many intervals, both in $p^*(W)$ and in $p^*(U)$, in verifying that this rather small collection is coherent.

We spend the remainder of this section verifying the utility of a coherent collection, giving some technical lemmas.

\begin{lemma}\label{gettingstarted}  Let $\{\coi(W_x, \iota_x, U_x)\}_{x\in X}$ be coherent and $x\in X$.

\begin{enumerate}

\item  Let $I \subseteq \pindex(W_x)$ be an interval and $I = I_0I_1 \cdots I_n$ where for each $0 \leq r \leq n$ we have an $x_r \in X$, an interval $I_r'$ in $\pindex(W_{x_j})$ and $i_r \in \{-1, 1\}$ such that $W_x \upharpoonright_p I_r \equiv (W_{x_r} \upharpoonright_p I_r')^{i_r}$.  Then

\begin{center}  $[[U_x \upharpoonright_p \varpropto(I, \iota_x)]] = \prod_{r = 0}^n [[(U_{x_r} \upharpoonright_p \varpropto(I_r', \iota_{x_r}))^{i_r}]]$
\end{center}

\noindent and letting $L = \{0 \leq r \leq n \mid |I_r| > 1\}$ we have

\begin{center}  $[[U_x \upharpoonright_p \varpropto(I, \iota_x)]] = \prod_{r \in L} [[(U_{x_r} \upharpoonright_p \varpropto(I_r', \iota_{x_r}))^{i_r}]]$.

\end{center}

\item  Let $I \subseteq \pindex(U_x)$ be an interval and $I = I_0I_1 \cdots I_n$ where for each $0 \leq r \leq n$ we have an $x_r \in X$, an interval $I_r'$ in $\pindex(U_{x_r})$ and $j_r \in \{-1, 1\}$ such that $U_x \upharpoonright_p I_r \equiv (W_{x_r} \upharpoonright_p I_r')^{j_r}$.  Then

\begin{center}  $[[W_x \upharpoonright_p \varpropto(I, \iota_x^{-1})]] = \prod_{r = 0}^n [[(W_{x_r} \upharpoonright_p \varpropto(I_r', \iota_{x_r}^{-1}))^{j_r}]]$
\end{center}

\noindent and letting $L = \{0 \leq r \leq n \mid |I_r| > 1\}$ we have

\begin{center}  $[[W_x \upharpoonright_p \varpropto(I, \iota_x^{-1})]] = \prod_{r \in L} [[(W_{x_r} \upharpoonright_p \varpropto(I_r', \iota_{x_r}^{-1}))^{j_r}]]$.

\end{center}

\end{enumerate}

\end{lemma}

\begin{proof}  The proof of each of (1) and (2) utilizes Lemma \ref{coilemma} and is essentially the same as that of \cite[Lemma 3.13]{Cor2}, word for word.
\end{proof}

\begin{definition}\label{commonrefinement}  Suppose that for a nonempty totally ordered set $\Lambda$ we have decompositions $\Lambda = I_0I_1 \cdots I_n$ and $\Lambda = I_0'I_1'\cdots I_k'$ into finitely many nonempty intervals.  We say the second decomposition is a \emph{refinement} of the first provided each $I_i$ is the union of some of the elements in $\{I_s'\}_{s = 0}^{k}$. The \emph{common refinement} of two decompositions $\Lambda = I_0I_1 \cdots I_n = I_0'I_1'\cdots I_k'$ is the finite decomposition $\Lambda = I_0''I_1'' \cdots I_{\ell}''$ where each $I_t''$ is a nonempty intersection of an element in $\{I_r\}_{r = 0}^n$ and an element in $\{I_s'\}_{s = 0}^{k}$.
\end{definition}

Note that the common refinement is a refinement of the two decompositions.  Also, if $W \in \Pfine(\{W_x\}_{x \in X})$ and $\pindex(W) = I_0 \cdots I_n$ is a decomposition as in Lemma \ref{elementsofthegeneratedsubgroup}, then any refinement of this decomposition also satisfies the conclusion of Lemma \ref{elementsofthegeneratedsubgroup}.  Similarly for a decomposition $\pindex(U) = I_0 \cdots I_n$ for $U \in\Pfine(\{U_x\}_{x \in X})$ as in Lemma \ref{elementsofthegeneratedsubgroupC}.

\begin{lemma}\label{passtorefinement}  Suppose $\{\coi(W_x, \iota_x, U_x)\}_{x \in X}$ is coherent and $W \in \Pfine(\{W_x\}_{x \in X})$.  Let $I_0, \ldots, I_n$ be a finite set of subintervals of $p^*(W)$ as in Lemma \ref{elementsofthegeneratedsubgroup} and $J = \{0 \leq r \leq n \mid |I_r| > 1\}$.  For each $r \in J$ select an $x_r \in X$, interval $\Lambda_r \subseteq p^*(W_{x_r})$ and $i_r \in \{-1, 1\}$ such that $W \upharpoonright_p I_r \equiv (W_{x_r} \upharpoonright_p \Lambda_r)^{i_r}$.  Let also $p^*(W) = I_0'\cdots I_k'$ be a refinement of the decomposition $p^*(W) = I_0\cdots I_n$ and $J' = \{0 \leq s \leq k \mid |I_s'| > 1\}$.  For each $s \in J'$ select an $x_s' \in X$, interval $\Lambda_s' \subseteq p^*(W_{x_{s}'})$ and $i_s'$ such that $W \upharpoonright_p I_r' \equiv (W_{x_s'} \upharpoonright_p \Lambda_s')^{i_s'}$.  Then

\begin{center}
$\prod_{r \in J}[[(U_{x_r} \upharpoonright_p \varpropto(\Lambda_r, \iota_{x_r}))^{i_r}]] = \prod_{s \in J'}  [[(U_{x_s'}\upharpoonright_p \varpropto(\Lambda_s', \iota_{x_s'}))^{i_s'}]]$.
\end{center}
\end{lemma}

\begin{proof}
For each $r \in J$ let $(J')_r = \{s \in J' \mid I_r \supseteq I_s'\}$.  By Lemma \ref{gettingstarted} (1) we know that

\begin{center}
$[[(U_{x_r} \upharpoonright_p \varpropto(\Lambda_r, \iota_{x_r}))^{i_r}]] = \prod_{s \in (J')_r}[[(U_{x_s'}\upharpoonright_p \varpropto(\Lambda_s', \iota_{x_s'}))^{i_s'}]]$

\end{center}

\noindent for each $r \in J$ (this is clear when $i_r = 1$, and when $i_r = -1$ the words are inverted twice).  Therefore we obtain

$$
\begin{array}{ll}
\prod_{r \in J} [[(U_{x_r} \upharpoonright_p \varpropto(\Lambda_r, \iota_{x_r}))^{i_r}]] & = \prod_{r \in J} \prod_{s \in (J')_r}[[(U_{x_s'}\upharpoonright_p \varpropto(\Lambda_s', \iota_{x_s'}))^{i_s'}]]\\
& =  \prod_{s \in J'} [[(U_{x_s'}\upharpoonright_p \varpropto(\Lambda_s', \iota_{x_s'}))^{i_s'}]].
\end{array}
$$
\end{proof}

Of course we also have the following.

\begin{lemma}\label{passtorefinementbackwards}  Suppose $\{\coi(W_x, \iota_x, U_x)\}_{x \in X}$ is coherent and $U \in \Pfine(\{U_x\}_{x \in X})$  Let $I_0, \ldots, I_n$ be a finite set of subintervals of $p^*(U)$ as in Lemma \ref{elementsofthegeneratedsubgroupC} and $J = \{0 \leq r \leq n \mid |I_r| > 1\}$.  For each $r \in J$ select an $x_r \in X$, interval $\Theta_r \subseteq p^*(U_{x_r})$ and $j_r \in \{-1, 1\}$ such that $U \upharpoonright_p I_r \equiv (U_{x_r} \upharpoonright_p \Theta_r)^{j_r}$.  Let also $p^*(U) = I_0'\cdots I_k'$ be a refinement of the decomposition $p^*(U) = I_0\cdots I_n$ and $J' = \{0 \leq s \leq k \mid |I_s'| > 1\}$.  For each $s \in J'$ select an $x_s' \in X$, interval $\Theta_s' \subseteq p^*(U_{x_s'})$ and $j_s'$ such that $U \upharpoonright_p I_r' \equiv (U_{x_s'} \upharpoonright_p \Theta_s')^{j_s'}$.  Then

\begin{center}
$\prod_{r \in J}[[(W_{x_r} \upharpoonright_p \varpropto(\Theta_r, \iota_{x_r}^{-1}))^{j_r}]] = \prod_{s \in J'}  [[(W_{x_s'}\upharpoonright_p \varpropto(\Theta_s', \iota_{x_s'}^{-1}))^{j_s'}]]$.
\end{center}
\end{lemma}

\begin{proof}
The argument is essentially the same as that of Lemma \ref{passtorefinement}, using part (2) of Lemma \ref{gettingstarted} instead of part (1).
\end{proof}

\begin{lemma}\label{welldefined}
Let $\{\coi(W_x, \iota_x, U_x)\}_{x\in X}$ be coherent.  By selecting for each $W \in \Pfine(\{W_x\}_{x \in X})$ a finite set of subintervals $I_0, \ldots, I_n$ of $\pindex(W)$ as in the conclusion of Lemma \ref{elementsofthegeneratedsubgroup}, letting $J = \{0 \leq r \leq n \mid |I_r| > 1\}$, selecting for each $r \in J$ an element $x_r \in X$, $i_r \in \{-1, 1\}$, and interval $\Lambda_r \subseteq \pindex(W_{x_r})$ such that $W \upharpoonright_p I_r \equiv (W_{x_r} \upharpoonright_p \Lambda_r)^{i_r}$ we obtain a function

\begin{center}
$\phi_0:  \Pfine(\{W_x\}_{x \in X}) \rightarrow \beth_{c}(\Pfine(\{U_x\}_{x \in X}))$
\end{center}

\noindent given by $\phi_0(W) = \prod_{r \in J} [[(U_{x_r} \upharpoonright_p \varpropto(\Lambda_r, \iota_{x_r}))^{i_r}]]$, whose definition is independent of the choices made of the set of subintervals $I_0, \ldots, I_n$, elements $x_r \in X$ and $i_r \in \{-1, 1\}$, and intervals $\Lambda_r \subseteq \pindex(W_{x_r})$.  The comparable map 

\begin{center}

$\phi_1: \Pfine(\{U_x\}_{x \in X}) \rightarrow \beth_{a, b}(\Pfine(\{W_x\}_{x \in X}))$

\end{center}

\noindent is also well-defined (i.e. independent of the various selections made).

\end{lemma}

\begin{proof}  Take a decomposition $p^*(W) = I_0\cdots I_n$ as in Lemma \ref{elementsofthegeneratedsubgroup}, and choices of $x_r \in X$ and $i_r \in \{-1, 1\}$ and interval $\Lambda_r \subseteq \pindex(W_{x_r})$ for elements in $J = \{0 \leq r \leq n \mid |I_r| > 1\}$ such that $W \upharpoonright_p I_r \equiv (W_{x_r} \upharpoonright_p \Lambda_r)^{i_r}$.  Take a possibly distinct choice of intervals $\pindex(W) = I_0'\cdots I_{n'}'$, and of $x_s' \in X$, of $i_s' \in \{-1, 1\}$, and $\Lambda_s' \subseteq \pindex(W_{x_s'})$ for elements in $J' = \{0 \leq s \leq n' \mid |I_s'| > 1\}$ such that $W \upharpoonright_p I_s' \equiv (W_{x_s'} \upharpoonright_p \Lambda_s')^{i_s'}$.  We pass to the mutual refinement of the decompositions $I_0\cdots I_n$ and $I_0'\cdots I_{n'}'$, say $\pindex(W) = I_0''\cdots I_{n''}''$ and make a third selection of elements of $x_t'' \in X$, exponents $i_t'' \in \{-1, 1\}$, and $\Lambda_t'' \subseteq \pindex(W_{x_t})$ satisfying $W\upharpoonright_p I_t'' \equiv (W_{x_t} \upharpoonright_p \Lambda_t'')^{i_t''}$ for elements of $J'' = \{0 \leq t \leq n'' \mid |I_t''| > 1\}$.  By applying Lemma \ref{passtorefinement} twice we have

$$
\begin{array}{ll}
\prod_{r \in J} [[(U_{x_r} \upharpoonright_p \varpropto(\Lambda_r, \iota_{x_r}))^{i_r}]]
& =  \prod_{t \in J''} [[(U_{x_t''}\upharpoonright_p \varpropto(\Lambda_t'', \iota_{x_t''}))^{i_t''}]]\\
& = \prod_{s \in J'} [[(U_{x_s} \upharpoonright_p \varpropto(\Lambda_s', \iota_{x_r}))^{i_s'}]]
\end{array}
$$

\noindent and so the function $\phi_0$ is well-defined.  The check of the well-definedness of the function $\phi_1$ is entirely analogous.
\end{proof}

\begin{lemma}\label{homomorphism}  The functions $\phi_0$ and $\phi_1$ from Lemma \ref{welldefined} are homomorphisms.  If $\Pfine(\{W_x\}_{x \in X}) = \Red_{a, b}$ and $\Pfine(\{U_x\}_{x \in X}) = \Red_c$ then $\phi_0$ and $\phi_1$ descend to isomorphisms

\begin{center}
$\Phi_0: \Red_{a, b}/\langle\langle \Pure_{a, b}\rangle\rangle \rightarrow \Red_c/\langle\langle \Pure_c \rangle\rangle$
\end{center}

\noindent and

\begin{center}
$\Phi_1: \Red_c/\langle\langle \Pure_c \rangle\rangle \rightarrow \Red_{a, b}/\langle\langle \Pure_{a, b}\rangle\rangle$
\end{center}

\noindent such that $\Phi_0$ is the inverse of $\Phi_1$. 
\end{lemma}

\begin{proof}
Suppose first that $W \in \Pfine(\{W_x\}_{x \in X})$ and $W \equiv W_0W_1$.  Let 

\[
W_0' = \left\{
\begin{array}{ll}
W_0
                                            & \text{if } \max(\pindex(W_0))\text{ does not exist}, \\
W_0 \upharpoonright_p \pindex(W_0) \setminus \{\max(\pindex(W_0))\}                             & \text{otherwise}
\end{array}
\right.
\]

\noindent and

\[
W_1' = \left\{
\begin{array}{ll}
W_1
                                            & \text{if } \min(\pindex(W_0))\text{ does not exist}, \\
W_1 \upharpoonright_p \pindex(W_1) \setminus \{\min(\pindex(W_1))\}                             & \text{otherwise}.
\end{array}
\right.
\]

\noindent We shall show that $\phi_0(W) = \phi_0(W_0')\phi_0(W_1')$ (it is clear that both $W_0'$ and $W_1'$ are in $\Pfine(\{W_x\}_{x \in X})$).  Note that we can decompose $\pindex(W)$ as a concatenation

\begin{itemize}

\item $\pindex(W) = \pindex(W_0')\overline{I_0}\underline{I_1}\pindex(W_1')$ in case $|\pindex(W) \setminus (\pindex(W_0') \cup \pindex(W_1'))| = 2$;

\item $\pindex(W) = \pindex(W_0')\overline{I_0}\pindex(W_1')$ in case $|\pindex(W) \setminus (\pindex(W_0') \cup \pindex(W_1'))| = 1$; or

\item $\pindex(W) = \pindex(W_0')\pindex(W_1')$ in case $|\pindex(W) \setminus (\pindex(W_0') \cup \pindex(W_1'))| = 0$

\end{itemize}

\noindent with $\overline{I_0}$ and $\underline{I_1}$ an interval having one element.  Take a decomposition $\pindex(W) = I_0 \cdots I_n$ as in Lemma \ref{elementsofthegeneratedsubgroup}, which is without loss of generality a refinement of the above decomposition.

Let $J = \{0 \leq r \leq n \mid |I_r| > 1\}$ and select for each $r \in J$ a $x_r \in X$, $i_r \in \{-1, 1\}$, and interval $\Lambda_r \subseteq \pindex(W_{x_r})$ such that $W \upharpoonright_p I_r \equiv (W_{x_r} \upharpoonright_p \Lambda_r)^{i_r}$.  Let $k_0$ be maximal such that $\bigcup I_{k_0} \subseteq \overline{W_0}$ and $k_1$ be minimal such that $\bigcup I_{k_1} \subseteq \overline{W_1}$.   Now

$$
\begin{array}{ll}
\phi_0(W) & = \prod_{r \in J} [[(U_{x_r} \upharpoonright_p \varpropto(\Lambda_r, \iota_{x_r}))^{i_r}]]\\
& =  \prod_{r \in J, r \leq k_0 }[[(U_{x_r} \upharpoonright_p \varpropto(\Lambda_r, \iota_{x_r}))^{i_r}]] \prod_{r \in J, r \geq k_1}[[(U_{x_r} \upharpoonright_p \varpropto(\Lambda_r, \iota_{x_r}))^{i_r}]]\\
& = \phi_0(W_0')\phi_0(W_1')
\end{array}
$$

Next we note the obvious fact that if $W \in \Pfine(\{W_x\}_{x \in X})$, then $\phi_0(W^{-1}) = (\phi_0(W))^{-1}$.  This is clear since if $\pindex(W) = I_0 \cdots I_n$ is a decomposition as in Lemma \ref{elementsofthegeneratedsubgroup} then so too is $\pindex(W^{-1}) = I_n^{-1} \cdots I_0^{-1}$, and defining $J$ as usual and selecting $x_r \in X$, etc. we write the equality

$$
\begin{array}{ll}
(\phi_0(W))^{-1} & = (\prod_{r \in J} [[(U_{x_r} \upharpoonright_p \varpropto(\Lambda_r, \iota_{x_r}))^{i_r}]])^{-1}\\
& =  \prod_{r \in J^{-1}} [[(U_{x_r} \upharpoonright_p \varpropto(\Lambda_r, \iota_{x_r}))^{-i_r}]]\\
& = \phi_0(W^{-1})
\end{array}
$$

\noindent where $J^{-1}$ is the set $J$ under the reverse order.

To finish the proof that $\phi_0$ is a homomorphism, we combine the information extracted so far.  Given arbitrary $W_0, W_1 \in \Pfine(\{W_x\}_{x \in X})$, by Lemma \ref{reduced} select words $V_{0, 0}, V_{0, 1}, V_{1, 0}, V_{1, 1}$ in $\Red_{a, b}$ such that

\begin{itemize}

\item $W_0 \equiv V_{0, 0}V_{0, 1}$;

\item $W_1 \equiv V_{1, 0}V_{1, 1}$;

\item $V_{0, 1} \equiv V_{1, 0}^{-1}$;

\item $V_{0, 0}V_{1, 1}$ is reduced.

\end{itemize}

Define

\[
V_{0, 0}' = \left\{
\begin{array}{ll}
V_{0, 0}
                                            & \text{if } \max(\pindex(V_{0, 0}))\text{ does not exist}, \\
V_{0, 0} \upharpoonright_p \pindex(V_{0, 0}) \setminus \{\max(\pindex(V_{0, 0}))\}  & \text{otherwise}
\end{array}
\right.
\]

\[
V_{1, 0}' = \left\{
\begin{array}{ll}
V_{1, 0}
                                            & \text{if } \max(\pindex(V_{1, 0}))\text{ does not exist}, \\
V_{1, 0} \upharpoonright_p \pindex(V_{1, 0}) \setminus \{\max(\pindex(V_{1, 0}))\}  & \text{otherwise}
\end{array}
\right.
\]

\[
V_{0, 1}' = \left\{
\begin{array}{ll}
V_{0, 1}
                                            & \text{if } \min(\pindex(V_{0, 1}))\text{ does not exist}, \\
V_{0, 1} \upharpoonright_p \pindex(V_{0, 1}) \setminus \{\min(\pindex(V_{0, 1}))\}  & \text{otherwise}
\end{array}
\right.
\]

\[
V_{1, 1}' = \left\{
\begin{array}{ll}
V_{1, 1}
                                            & \text{if } \min(\pindex(V_{1, 1}))\text{ does not exist}, \\
V_{1, 1} \upharpoonright_p \pindex(V_{1, 1}) \setminus \{\min(\pindex(V_{1, 1}))\}  & \text{otherwise}.
\end{array}
\right.
\]

\noindent Note that $V_{0, 1}' \equiv (V_{1, 0}')^{-1}$.  We have

$$
\begin{array}{ll}
\phi_0(W_0)\phi_0(W_1) & = \phi_0(V_{0, 0}')\phi_0(V_{0, 1}')\phi_0(V_{1, 0}')\phi_0(V_{1, 1}')\\
& = \phi_0(V_{0, 0}')\phi_0(V_{1, 1}')\\
& = \phi_0(V_{0, 0}V_{1, 1})\\
& = \phi_0(\Red(W_0W_1))
\end{array}
$$

\noindent and the proof that $\phi_0$ is a homomorphism is complete.  The proof that $\phi_1$ is a homomorphism follows via mutatis mutandis.

Now suppose that $\Pfine(\{W_x\}_{x \in X} = \Red_{a, b}$ and $\Pfine(\{U_x\}_{x \in X}) = \Red_c$.  To see that $\Phi_0$ can be well-defined, note that each pure word $W \in \Pure_{a, b}$ has $\phi_0(W) = [[E]]$ since $\pindex(W)$ consists of a single element.  So $\ker(\beth_{a, b}) \leq \ker(\phi_1)$ and we obtain a homomorphism $\Phi_0$ as desired, and $\Phi_1$ is similarly obtained.

To see that $\Phi_0$ and $\Phi_1$ are inverse isomorphisms we let $W \in \Red_{a, b}$ be given, take $\pindex(W) \equiv I_0\cdots I_n$ as in Lemma \ref{elementsofthegeneratedsubgroup}, define $J$ as usual, and select all other parameters.  We get

$$
\begin{array}{ll}
\Phi_1\circ\Phi_0([[W]]) & = \Phi_1(\prod_{r \in J} [[(U_{x_r} \upharpoonright_p \varpropto(\Lambda_r, \iota_{x_r}))^{i_r}]])\\
& = \prod_{r \in J} \Phi_1([[(U_{x_r} \upharpoonright_p \varpropto(\Lambda_r, \iota_{x_r}))^{i_r}]])\\
& = \prod_{r \in J}[[(W_{x_r} \upharpoonright_p \varpropto(\Lambda_r, \iota_{x_r}))^{i_r}]]\\
& = [[W]]
\end{array}
$$

\noindent where the last line holds because the unreduced word $\prod_{r \in J}(W_{x_r} \upharpoonright_p \varpropto(\Lambda_r, \iota_{x_r}))^{i_r}$ is obtained from $W$ by removing finitely many points in $\pindex(W)$.  So $\Phi_1 \circ \Phi_0$ is identity, and by the same reasoning $\Phi_0 \circ \Phi_1$ is also identity.

\end{proof}

\end{section}

\begin{section}{Extending coherent collections in discrete cases} \label{discrete}

Now that Lemma \ref{homomorphism} provides us a means of creating the desired isomorphism, we give a sequence of lemmas which demonstrate how to keep extending a coherent collection of coi triples.  We begin with some small steps, showing that nesting coherent collections behave well and that it is possible to make certain conservative extensions of coherent collections (cf. Lemma 3.12, 3.17, 3.18 of \cite{Cor2} respectively).

\begin{lemma}\label{ascendingchaincoi}  Suppose that $(T, \leq)$ is a totally ordered set, $\{C_t\}_{t \in T}$ is a set of coherent collections of coi triples such that $t \leq t'$ implies $C_t \subseteq C_{t'}$.  Then $\bigcup_{t \in T} C_t$ is coherent.
\end{lemma}

\begin{proof}
Suppose $\coi(W, \iota, U), \coi(W', \iota', U') \in \bigcup_{t \in T} C_t$ and intervals $I_0 \subseteq \pindex(W)$ and $I_1 \subseteq \pindex(W')$ and $i\in \{-1, 1\}$ are such that $W \upharpoonright_p I_0 \equiv (W' \upharpoonright_p I_1)^i$, we select $t \in T$ such that $\coi(W, \iota, U), \coi(W', \iota', U') \in C_t$.  Since $C_t$ is coherent we know

\begin{center}

$[[U \upharpoonright_p \varpropto(I_0, \iota)]] = [[(U' \upharpoonright_p \varpropto(I_1, \iota'))^i]]$

\end{center}

\noindent  The check when $\coi(W, \iota, U), \coi(W', \iota', U') \in \bigcup_{t \in T} C_t$ have intervals $I \subseteq \pchunk(U)$ and $I' \subseteq \pchunk(U')$ and $j \in \{-1, 1\}$ with $U \upharpoonright_p I \equiv (U' \upharpoonright_p I')^j$ is analogous, so the proof is finished.
\end{proof}

\begin{definition}\label{standardized}  We will say that a collection of coi triples $C$ is \emph{standardized} if $$\{\coi(a_n, \iota_{a_n}, E)\}_{n \in \omega} \cup \{\coi(b_n, \iota_{b_n}, E)\}_{n \in \omega} \cup \{\coi(E, \iota_{c_m}, c_m)\}_{m \in \omega} \subseteq C$$ where each $\iota_{a_n}, \iota_{b_n}, \iota_{c_m}$ is the empty function.  
\end{definition}

\begin{remark}
It is clear that any coherent coi collection $C$ we can be extended to a standardized $C'$ which will also be coherent, simply by adding to $C$ the countably many coi triples listed in the definition.  The check that this new $C'$ is coherent simply involves noticing that expressions are equal to $[[E]]$.  It will be be convenient to assume that our collections are standardized.
\end{remark}

\begin{lemma} \label{findsomerepresentative}  Let $\{\coi(W_x, \iota_x, U_x)\}_{x \in X}$ be coherent and standardized.

\begin{enumerate}

\item  If $W \in \Pfine(\{W_x\}_{x \in X})$ there exists $U \in \Pfine(\{U_x\}_{x \in X})$ and coi $\iota$ from $W$ to $U$ such that $\{\coi(W_x, \iota_x, U_x)\}_{x \in X} \cup \{\coi(W, \iota, U)\}$ is coherent.

\item If $U \in \Pfine(\{U_x\}_{x \in X})$ then there exists $W \in \Pfine(\{W_x\}_{x \in X})$ and coi $\iota$ from $W$ to $U$ such that $\{\coi(W_x, \iota_x, U_x)\}_{x \in X} \cup \{\coi(W, \iota, U)\}$ is coherent. 

\end{enumerate}

\end{lemma}

\begin{proof}  
We begin with the proof of (1), so let $W \in \Pfine(\{W_x\}_{x \in X})$.  If $W \equiv E$ then let $U \equiv E$ and $\iota = \emptyset$; the coherence of the collection $\{\coi(W_x, \iota_x, U_x)\}_{x \in X} \cup \{\coi(W, \iota, U)\}$ is quite clear.  If $W$ is not the empty word then let $\pindex(W) = I_0\cdots I_n$ be a decomposition into intervals as in Lemma \ref{elementsofthegeneratedsubgroup}.  Write $J = \{0 \leq j \leq n \mid |I_j| > 1\}$, for each $j \in J$ select $x_j \in X$ and $i_j \in \{-1, 1\}$ and interval $\Lambda_j \subseteq \pindex(W_{x_j})$ with $W \upharpoonright_p I_j \equiv (W_{x_j} \upharpoonright_p \Lambda_j)^{i_j}$.  Let $J' = \{j \in J \mid (U_{x_j}\upharpoonright \varpropto(\Lambda_j, \iota_{x_j}))^{i_j} \not\equiv E\}$.  If $J' = \emptyset$ then $\pindex(W)$ is finite (Lemma \ref{coilemma2}) and we can take $U \equiv E$ and $\iota = \emptyset$, and the check for coherence is easy.

Assume $J' \neq \emptyset$ and enumerate $J' = \{j_0, \ldots, j_{|J'| - 1}\}$ in increasing order.  For each $j \in J'$ let 

\begin{enumerate}[(a)]

\item $f_j: I_j \rightarrow \Lambda_j^{i_j}$ be an order isomorphism witnessing that $W \upharpoonright_p I_j \equiv (W_{x_j} \upharpoonright_p \Lambda_j)^{i_j}$.;

\item $U_j' \equiv (U_{x_j}\upharpoonright_p \varpropto(\Lambda_j, \iota_{x_j}))^{i_j}$;

\item $U_j \equiv  U_j'c_m$ if

\begin{enumerate}[(i)]

\item $j = j_t < j_{|J'| - 1}$,

\item both $\max(\pindex(U_{j_t}'))$ and $\min(\pindex(U_{j_{t + 1}}'))$ exist, and

\item $U_{j_t}' \upharpoonright_p \{\max(\pindex(U_{j_t}'))\}$ and $U_{j_{t + 1}}' \upharpoonright_p \{\min(\pindex(U_{j_{t + 1}}'))\}$ are each $m'$-pure and $m = m' + 1$

\end{enumerate}

\noindent and otherwise let $U_j \equiv U_j'$.

\end{enumerate}

\noindent Let $U \equiv \prod_{j \in J'} U_j$.  It is easy to see (by how the $U_j$ were chosen) that $U \in \Red_c$, and in fact $U \in \Pfine(\{U_x\}_{x \in X})$ since the collection is standardized.  Letting $\Lambda \subseteq \pindex(W)$ be given by $\Lambda = \bigcup_{j \in J} f_j^{-1}(\Lambda_j \cap \dom(\iota_{x_j}))$ it is clear by Lemma \ref{basiccloseproperties} that $\Close(\Lambda, \pindex(W))$.  Let $\Theta \subseteq \pindex(U)$ be given by $\Theta = \bigcup_{j \in J} \iota_{x_j}(\Lambda_j \cap \dom(\iota_{x_j}))$ and again $\Close(\Theta, \pindex(U))$.  Let $\iota: \Lambda \rightarrow \Theta$ be given by $\iota \upharpoonright f_j^{-1}(\Lambda_j \cap \dom(\iota_{x_j})) = \iota_{x_j} \circ f_j \upharpoonright  f_j^{-1}(\Lambda_j \cap \dom(\iota_{x_j}))$ for each $j \in J$.  Clearly $\iota$ is an order isomorphism (when $i_j = -1$ the composition reverses the order twice).

Now we claim that $\{\coi(W_x, \iota_x, U_x)\}_{x \in X} \cup \{\coi(W, \iota, U)\}$ is coherent, which we must check.  Suppose that $I \subseteq \pindex(W)$, $x \in X$, $I' \subseteq \pindex(W_x)$ and $i \in \{-1, 1\}$ are such that $W \upharpoonright_p I \equiv (W_x \upharpoonright_p I')^i$.  Note that $I = I$ is a decomposition of $\pindex(W \upharpoonright_p I) = I$ as in Lemma \ref{elementsofthegeneratedsubgroup}, $W \upharpoonright_p I \equiv (W_x \upharpoonright_p I')^i$, and $I = (I \cap I_0)(I \cap I_1) \cdots (I \cap I_n)$ since $W \upharpoonright_p I \equiv (W_x \upharpoonright_p I')^i$ is a refinement of it.  Thus

$$
\begin{array}{ll}
[[U \upharpoonright_p \varpropto (I, \iota)]] & = \prod_{j \in J'}[[(U_{x_j} \upharpoonright_p \varpropto(\iota_{x_j}, I \cap I_j))^{i_j}]]\\
& = \prod_{j \in J}[[(U_{x_j} \upharpoonright_p \varpropto(\iota_{x_j}, I \cap I_j))^{i_j}]]\\
& = [[(U_x \upharpoonright_p \varpropto(\iota_x, I'))^i]]
\end{array}
$$

\noindent where the first equality is clear, the second is by adding a copy of $[[E]]$ to  for each $j \in J \setminus J'$, and the third equality is by Lemma \ref{passtorefinement}.  On the other hand if intervals $I, I' \subseteq \pindex(W)$  and $i \in \{-1, 1\}$ are such that $W \upharpoonright_p I \equiv (W \upharpoonright_p I')^i$ we write

$$
\begin{array}{ll}
[[U \upharpoonright_p \varpropto (I, \iota)]] & = \prod_{j \in J'}[[(U_{x_j} \varpropto(\iota_{x_j}, I \cap I_j))^{i_j}]]\\
& = \prod_{j \in J}[[(U_{x_j} \varpropto(\iota_{x_j}, I \cap I_j))^{i_j}]]\\
& = \phi_0(W \upharpoonright_p I)\\
& =  (\prod_{j \in J}[[(U_{x_j} \varpropto(\iota_{x_j}, I' \cap I_j))^{i_j})^i]]\\
& = (\prod_{j \in J'}[[(U_{x_j} \varpropto(\iota_{x_j}, I \cap I_j))^{i_j})^i]]\\
& = [[(U \upharpoonright_p \varpropto (I', \iota))^i]]
\end{array}
$$

Now suppose that $I \subseteq \pindex(U)$, $x \in X$, $I' \subseteq \pindex(U_x)$ and $i \in \{-1, 1\}$ are such that $U \upharpoonright_p I \equiv (U_x \upharpoonright_p I')^i$.  Then

$$
\begin{array}{ll}
[[W \upharpoonright_p \varpropto (I, \iota^{-1})]] & = \prod_{k = 0}^n [[W \upharpoonright_p I_k \cap \varpropto (I, \iota^{-1})]]\\
& = \prod_{j \in J}[[W \upharpoonright_p I_j \cap \varpropto (I, \iota^{-1})]]\\
& = \prod_{j \in J}[[(W_{x_j} \upharpoonright_p \varpropto (I_j'', \iota_{x_j}^{-1}))^{i_j}]]\\
& = \phi_1(U \upharpoonright_p I)\\
& = (\phi_1(U_x \upharpoonright_p I'))^i\\
& = [[(W_x \upharpoonright_p \varpropto(I', \iota_x^{-1}))^i]]
\end{array}
$$

\noindent where the first equality is clear, the second comes from deleting finitely many expressions which are equal to $[[E]]$, the third is an application of Lemma \ref{passtorefinementbackwards} (we are now using the fact that $\{c_m\}_{m \in \omega} \subseteq \Pfine(\{U_x\}_{x \in X})$, and the interval $I_j'' \subseteq \pindex(U_{x_j})$ is the appropriate one), the remaining are by how $\phi_1$ is defined.  If on the other hand $I, I' \subseteq \pindex(U)$ and $i \in \{-1, 1\}$ are such that $U \upharpoonright_p I \equiv (U \upharpoonright_p I')^i$ then

$$
\begin{array}{ll}
[[W \upharpoonright_p \varpropto (I, \iota^{-1})]] & = \prod_{j \in J}[[(W_{x_j} \upharpoonright_p \varpropto(I_j'', \iota_{x_j}^{-1}))^{i_j}]]\\
& = \phi_1(U \upharpoonright_p I)\\
& = (\phi_1(U \upharpoonright_p I'))^i\\
& = (\prod_{j \in J}[[(W_{x_j} \upharpoonright_p \varpropto(I_j''', \iota_{x_j}^{-1}))^{i_j}]])^i\\
& = [[(W \upharpoonright_p \varpropto (I', \iota^{-1}))^i]]
\end{array}
$$

\noindent where the intervals $I_j'', I_j''' \subseteq \pindex(U_{x_j})$ are appropriate.  So coherence holds and (1) is proved.

For (2) we will construct $W \in \Pfine(\{W_x\}_{x \in X})$ and $\iota$ and the check for coherence will be omitted as it is almost the same as in case (1).  Letting $U \in \Pfine(\{U_x\}_{x \in X})$, the claim is trivial if $U \equiv E$, so let $\pindex(U) = I_0\cdots I_n$ as in Lemma \ref{elementsofthegeneratedsubgroup}.  As in (1) let $J = \{0 \leq j \leq n \mid |I_j| > 1\}$, select $x_j \in X$ and $i_j \in \{-1. 1\}$ and interval $\Theta_j \subseteq \pindex(U_{x_j})$ with $U \upharpoonright_p I_j \equiv (U_{x_j}\upharpoonright_p \Theta_j)^i$ for each $j \in J$.  Let $J' = \{j \in J \mid (W_{x_j}\upharpoonright \varpropto(\Theta_j, \iota_{x_j}^{-1}))^{i_j} \not\equiv E\}$.  If $J' = \emptyset$ then $\pindex(U)$ is finite and let $W \equiv E$ and $\iota = \emptyset$.

Assuming $J' \neq \emptyset$ and enumerate $J' = \{j_0, \ldots, j_{|J'| - 1}\}$ in increasing order.  For each $j \in J'$ let 

\begin{enumerate}[(a)]

\item $f_j: I_j \rightarrow \Theta_j^{i_j}$ be an order isomorphism witnessing that $W \upharpoonright_p I_j \equiv (W_{x_j} \upharpoonright_p \Theta_j)^{i_j}$.;

\item $W_j' \equiv (W_{x_j}\upharpoonright_p \varpropto(\Theta_j, \iota_{x_j}^{-1}))^{i_j}$;

\item $W_j \equiv  W_j'V_j$ where

\begin{enumerate}[(i)]

\item $V_j \equiv a_0$ if 

\begin{itemize}

\item[$(*_1)$] $j = j_t < j_{|J'| - 1}$,

\item[$(*_2)$]  both $\max(\pindex(U_{j_t}'))$ and $\min(\pindex(U_{j_{t + 1}}'))$ exist, and

\item[$(*_3)$] $U_{j_t}' \upharpoonright_p \{\max(\pindex(U_{j_t}'))\}$ and $U_{j_{t + 1}}' \upharpoonright_p \{\min(\pindex(U_{j_{t + 1}}'))\}$ are each $b$-pure;

\end{itemize}

\item $V_j \equiv b_0$ if 

\begin{itemize}

\item[$(\circ_1)$] $j = j_t < j_{|J'| - 1}$,

\item[$(\circ_2)$]  both $\max(\pindex(U_{j_t}'))$ and $\min(\pindex(U_{j_{t + 1}}'))$ exist, and

\item[$(\circ_3)$] $U_{j_t}' \upharpoonright_p \{\max(\pindex(U_{j_t}'))\}$ and $U_{j_{t + 1}}' \upharpoonright_p \{\min(\pindex(U_{j_{t + 1}}'))\}$ are each $a$-pure;

\end{itemize}

\item $V_j \equiv E$ otherwise.

\end{enumerate}

\end{enumerate}

Let $W \equiv \prod_{j \in J'} W_j$ and it is clear that $W \in \Red_{a, b}$, indeed $W \in \Pfine(\{W_x\}_{x \in X})$ as the collection is standardized.  Take $\Theta \subseteq \pindex(W)$ given by $\Theta = \bigcup_{j \in J} f_j^{-1}(\Theta_j \cap \ran(\iota_{x_j}))$, and $\Close(\Theta, \pindex(U))$.  Let $\Lambda \subseteq \pindex(W)$ be given by $\Lambda = \bigcup_{j \in J'} \iota_{x_j}^{-1}(\Theta_j \cap \ran(\iota_{x_j}))$ and we have $\Close(\Lambda, \pindex(W))$.  Let $\iota: \Lambda \rightarrow \Theta$ be given by $\iota \upharpoonright \iota_{x_j}^{-1}(\Theta_j \cap \ran(\iota_{x_j})) = \iota_{x_j} \upharpoonright  \iota_{x_j}^{-1}(\Theta_j \cap \ran(\iota_{x_j}))$ for each $j \in J$.

\end{proof}

\begin{definition}  Recall that for $W \in \W_{a, b}$ we let $d(W) = \mathfrak{n} \in \omega \cup \{\infty\}$ where if $W \equiv E$ then $\mathfrak{n} = \infty$ and if $W \not \equiv E$ then $\mathfrak{n}$ is the least subscript of any letter appearing in $W$.  For example, $d(a_2a_3a_4 \cdots) = 2$.  Define $d(U)$ for a word $U \in \W_c$ is precisely the same manner.
\end{definition}

We record the following strengthening of the previous lemma.

\begin{lemma}\label{smallrepresentative}  
 Let $\{\coi(W_x, \iota_x, U_x)\}_{x \in X}$ be coherent and standardized.  Let also $N \in\omega$ be given.

\begin{enumerate}

\item  If $W \in \Pfine(\{W_x\}_{x \in X})$ there exists $U \in \Pfine(\{U_x\}_{x \in X})$ and coi $\iota$ from $W$ to $U$ such that $\{\coi(W_x, \iota_x, U_x)\}_{x \in X} \cup \{\coi(W, \iota, U)\}$ is coherent.  Moreover $d(U) \geq N$, and additionally $\dom(\iota) \neq \emptyset$ provided $W \not\equiv E$.

\item If $U \in \Pfine(\{U_x\}_{x \in X})$ then there exists $W \in \Pfine(\{W_x\}_{x \in X})$ and coi $\iota$ from $W$ to $U$ such that $\{\coi(W_x, \iota_x, U_x)\}_{x \in X} \cup \{\coi(W, \iota, U)\}$ coherent.  Moreover $d(W) \geq N$, and additionally $\dom(\iota) \neq \emptyset$ provided $U \not\equiv E$.

\end{enumerate}

\end{lemma}

\begin{proof}
We first show (1).  Assume the hyootheses and let $W \in \Pfine(\{W_x\}_{x \in X})$ be given.  If $W \equiv E$ then take $U \equiv E$ and $\iota = \emptyset$.  If $W \not\equiv E$ and $\pindex(W)$ is finite, then take take $\lambda \in \pindex(W)$, let $U \equiv c_N$ and $\iota: \{\lambda\} \rightarrow \pindex(U)$ to be the only possible function.  It is now clear that $\{\coi(W_x, \iota_x, U_x)\}_{x \in X} \cup \{\coi(W, \iota, U)\}$ is coherent (one keeps seeing that group elements are equal to $[[E]]$).

Now suppose that $\pindex(W)$ is infinite.  Apply Lemma \ref{findsomerepresentative} (1) to obtain a $U' \in \Pfine(\{U_x\}_{x \in X})$ and coi $\iota'$ from $W$ to $U$ with $\{\coi(W_x, \iota_x, U_x)\}_{x \in X} \cup \{\coi(W, \iota', U')\}$ coherent.  For convenience write 

\begin{itemize}

\item $\Lambda = \pindex(W)$;

\item $\Theta = \pindex(U')$;

\item $U' \equiv_p \prod_{\theta \in \Theta} U_{\theta}'$.

\end{itemize}

\noindent We shall modify $U'$ only slightly to obtain $U$.  Let $\Theta_1 = \{\theta \in \Theta \mid d(U_{\theta}') < N\}$, and since $U \in \W_c$ we know that $\Theta_1$ is finite, so label $\Theta_1 = \{\theta_0, \ldots, \theta_n\}$ with the $\theta_j$ in increasing order in $\Theta$.  (If $\Theta_1 = \emptyset$ then let $U \equiv U'$ and $\iota = \iota'$ and we are already done.)  Select $m_0 \geq N$ such that $$(\prod_{\theta < \theta_0} U_{\theta}') c_{m_0} (\prod_{\theta_0 < \theta < \theta_1} U_{\theta}')$$ is given in pure decomposition (that is, there is no nonempty terminal subword of $\prod_{\theta < \theta_0} U_{\theta}'$ which is $m_0$-pure and there is no nonempty initial subword of $\prod_{\theta_0 < \theta < \theta_1} U_{\theta}'$ which is $m_0$-pure).  Select $m_1 \geq N$ such that $$(\prod_{\theta < \theta_0} U_{\theta}') c_{m_0} (\prod_{\theta_0 < \theta < \theta_1} U_{\theta}') c_{m_1} (\prod_{\theta_1 < \theta < \theta_2} U_{\theta}')$$ is given in pure decomposition, etc., and finally select $m_n \geq N$ so that $$(\prod_{\theta < \theta_0} U_{\theta}') c_{m_0} (\prod_{\theta_0 < \theta < \theta_1} U_{\theta}') \cdots c_{m_{n - 1}} (\prod_{\theta_{m_{n - 1} < \theta < \theta_{m_n}}} U_{\theta}') c_{m_n} (\prod_{\theta_{m_n} < \theta} U_{\theta}')$$ is in pure decomposition.  Of course it is quite easy to make such selections; generally we can even select $m_j \in \{N, N + 1, N + 2\}$.

Now let $U_{\theta} \equiv U_{\theta}'$ for $\theta \in \Theta \setminus \Theta_1$ and $U_{\theta} \equiv c_{m_j}$ for $\theta = \theta_j \in \Theta_1$.  Let $U \equiv \prod_{\theta \in \Theta} U_{\theta}$ and it is clear that $U$ is reduced and in fact $\pindex(U) \equiv \prod_{\theta \in \Theta} \pindex(U_{\theta})$ and $U \in \Pfine(\{U_x\}_{x \in X})$.  Let $\iota = \iota'$.  For any interval $I \subseteq \Lambda$ it is clear that $[[U \upharpoonright_p \varpropto (I, \iota)]] = [[U' \upharpoonright_p \varpropto (I, \iota')]]$, and similarly for any interval $I \subseteq \Theta$ we have $[[W \upharpoonright_p \varpropto(I, \iota^{-1})]] = [[W \upharpoonright_p \varpropto(I, (\iota')^{-1})]]$.  From this it is clear that $\{\coi(W_x, \iota_x, U_x)\}_{x \in X} \cup \{\coi(W, \iota', U')\} \cup \{\coi(W, \iota, U)\}$ is coherent, so $\{\coi(W_x, \iota_x, U_x)\}_{x \in X} \cup \{\coi(W, \iota, U )\}$ is coherent.

The proof of (2) is only slightly more complicated and we give the sketch.  In the nontrivial case where $\pindex(U)$ is infinite we select $\coi(W', \iota', U)$ so that $$\{\coi(W_x, \iota_x, U_x)\}_{x \in X} \cup \{\coi(W', \iota', U)\}$$ is coherent (Lemma \ref{findsomerepresentative} (2)) and define $\Lambda$ and $\Theta$ as before and write $W' \equiv_p \prod_{\lambda \in \Lambda} W_{\lambda}'$.  Let $\Lambda_1 = \{\lambda \in \Lambda \mid d(W_{\lambda}') < N\}$.  In the nontrivial case where $\Lambda_1 \neq \emptyset$ enumerate $\Lambda_1 = \{\lambda_0, \ldots, \lambda_n\}$ and define $V_0, \ldots, V_n$ as follows.  Each $V_j$ is either $a_Nb_N$, $a_N$, $b_N$, or $b_Na_N$ and is simply chosen so that the word $$W \equiv (\prod_{\lambda < \lambda_0} W_{\lambda}') V_0 (\prod_{\lambda_0 < \lambda < \lambda_1} W_{\lambda}') V_1 \cdots (\prod_{\lambda_{n - 1} < \lambda < \lambda_n} W_{\lambda}') V_n (\prod_{\lambda_n < \lambda} W_{\lambda}')$$ is reduced and also so that $\pindex(W) = I_0\pindex(V_0)I_1\pindex(V_1)\cdots I_{n}\pindex(V_n) I_{n + 1}$ where $I_0 = \{\lambda \in \Lambda \mid \lambda < \lambda_0\}$, $I_{n + 1} = \{\lambda \in \Lambda \mid \lambda_n < \lambda\}$ and $I_{j} = \{\lambda \in \Lambda \mid \lambda_{j - 1} < \lambda < \lambda_j\}$ for $0 < j < n + 1$.

Recall that $\Lambda = \pindex(W')$ and note that $\Lambda \setminus \pindex(W) = \Lambda_1$ and $\pindex(W) \setminus \Lambda = \bigcup_{j = 0}^n \pindex(V_j)$.  Define $\iota = \iota' \upharpoonright (\dom(\iota') \cap \pindex(W))$.  Clearly for an interval $I \subseteq \pindex(W)$ we have $[[U \upharpoonright_p \varpropto(I, \iota)]] = [[U \upharpoonright_p \varpropto(I', \iota')]]$ where $I' \subseteq \Lambda$ is the smallest interval including the set $I \cap \Lambda$, since the symmetric difference $\varpropto(I', \iota') \Delta \varpropto(I, \iota)$ is finite.  If $I \subseteq \Theta$ then

$$
\begin{array}{ll}
[[W' \upharpoonright_p \varpropto(I, (\iota')^{-1})]]  & = \prod_{\ell = 0}^{n + 1}[[W' \upharpoonright_p (I_{\ell} \cap\varpropto(I, (\iota')^{-1}))]]  \\
& = \prod_{\ell = 0}^{n + 1}[[W \upharpoonright_p (I_{\ell} \cap\varpropto(I, (\iota)^{-1}))]]  \\
& = [[W \upharpoonright_p \varpropto (I, \iota^{-1})]]
\end{array}
$$

\noindent and coherence of the enlarged coi collection follows.

\end{proof}

Now that we have passed some preliminaries we are ready to state and prove the following.

\begin{lemma}\label{typeomega}
Suppose that $\{\coi(W_x, \iota_x, U_x)\}_{x \in X}$ is coherent and standardized and that $|X| < 2^{\aleph_0}$.  Each of the following holds.

\begin{enumerate}

\item If $W \in \Red_{a, b} \setminus \Pfine(\{W_x\}_{x \in X})$ and $\pindex(W) = \prod_{n \in \omega} I_n$ with 

\begin{enumerate}[(i)]

\item each $I_n \neq \emptyset$; and

\item each $W \upharpoonright_p I_n \in \Pfine(\{W_x\}_{x \in X})$

\end{enumerate}

\noindent then there exists $U \in \Red_{c}$ and coi $\iota$ from $W$ to $U$ such that $$\{\coi(W_x, \iota_x, U_x)\}_{x \in X} \cup\{\coi(W \iota, U)\}$$ is coherent.

\item If $U \in \Red_{a, b} \setminus \Pfine(\{U_x\}_{x \in X})$ and $\pindex(U) = \prod_{n \in \omega} I_n$ with 

\begin{enumerate}[(i)]

\item each $I_n \neq \emptyset$; and

\item each $U \upharpoonright_p I_n \in \Pfine(\{U_x\}_{x \in X})$

\end{enumerate}

\noindent then there exists $W \in \Red_{a, b}$ and coi $\iota$ from $W$ to $U$ such that $$\{\coi(W_x, \iota_x, U_x)\}_{x \in X} \cup\{\coi(W, \iota, U)\}$$ is coherent.
\end{enumerate}
\end{lemma}

\begin{proof}[Proof of part (1)]
Write $W_n \equiv W \upharpoonright_p I_n$ for each $n \in \omega$.  By Lemma \ref{smallrepresentative} (1) select a nonempty $U_0 \in \Pfine(\{U_x\}_{x \in X})$ with $d(U_0) > 0$ and a coi $\iota_0$ from $W_0$ to $U_0$ such that $\dom(\iota_0) \neq \emptyset$ and $\{\coi(W_x, \iota_x, U_x)\}_{x \in X} \cup\{\coi(W_0, \iota_0, U_0)\}$ is coherent.  Assuming we have chosen $U_{\ell - 1}$ and $\iota_{\ell - 1}$, we choose by Lemma \ref{smallrepresentative} (1) a nonempty $U_{\ell} \in \Pfine(\{U_x\}_{x \in X})$ with $d(U_{\ell}) > \ell$ and coi $\iota_{\ell}$ from $W_{\ell}$ to $U_{\ell}$ such that $\dom(\iota_{\ell}) \neq \emptyset$ and $\{\coi(W_x, \iota_x, U_x)\}_{x \in X} \cup\{\coi(W_n, \iota_n, U_n)\}_{n = 0}^{\ell}$ is coherent.  By Lemma \ref{ascendingchaincoi} the collection $\{\coi(W_x, \iota_x, U_x)\}_{x \in X} \cup \{\coi(W_n, \iota_n, U_n)\}_{n \in \omega}$ is coherent.  Also it is clear that $\Pfine(\{W_x\}_{x \in X} \cup \{W_n\}_{n \in \omega}) = \Pfine(\{W_x\}_{x \in X})$ and $\Pfine(\{U_x\}_{x \in X} \cup \{U_n\}_{n \in \omega}) = \Pfine(\{U_x\}_{x \in X})$.

We will now choose a sequence $V_n$ of words, with either $V_n \equiv c_n$ or $V_n \equiv c_n^2$.  Note that for any such choice it is clear that the word $\prod_{n \in \omega} U_nV_n$ is reduced and moreover $\pindex(\prod_{n \in \omega} U_nV_n) = \prod_{n \in \omega} \pindex(U_n)\pindex(V_n)$.  Note that for any $V \in \Red_c$ and $m \in \omega$ there are only finitely many order monotonic functions (either strictly increasing or strictly decreasing) $f: \omega \setminus \{0, 1, \ldots, m - 1\} \rightarrow \pindex(U)$ such that 

\begin{itemize}

\item $V \upharpoonright_p \{f(r)\}$ is $r$-pure for each $r \in \omega \setminus \{0, 1, \ldots, m - 1\}$; and

\item $d(V \upharpoonright_p (\min\{f(r), f(r + 1)\}, \max\{f(r), f(r + 1)\})) > r + 1$.

\end{itemize}

\noindent Indeed, once one knows the value $f(m)$ and whether the function $f$ is increasing or decreasing, then the function is totally determined.  Therefore it is possible to select $\sigma: \omega \rightarrow \{1, 2\}$ such that the word $$U \equiv \prod_{n \in \omega} U_nc_n^{\sigma(n)}$$ has no nontrivial terminal p-chunk which is a p-chunk of some element in $\{U_x^{\pm 1}\}_{x \in X} \cup \{U_n^{\pm 1}\}_{n \in \omega}$ (from the fact that $|X| + |\omega| < 2^{\aleph_0}$).  Thus by Lemma \ref{elementsofthegeneratedsubgroupC} there is no nonempty terminal p-chunk of $U$ which is in the group $\Pfine(\{U_x\}_{x \in X})$ but since the collection $\{\coi(W_x, \iota_x, U_x)\}_{x \in X} \cup \{\coi(W_n, \iota_n, U_n)\}_{n \in \omega}$ is standardized we know that every proper initial p-chunk of $U$ is an element of $\Pfine(\{U_x\}_{x \in X})$.

Let $\iota = \bigcup_{n \in \omega} \iota_n$ and clearly this is a coi from $W$ to $U$, as applying Lemma \ref{basiccloseproperties} (iii) one sees that indeed $\Close(\dom(\iota)), \pindex(W))$ and $\Close(\ran(\iota), \pindex(U))$.  It now remains to check that the collection $$\{\coi(W_x, \iota_x, U_x)\}_{x \in X} \cup \{\coi(W_n, \iota_n, U_n)\}_{n \in \omega} \cup \{\coi(W, \iota, U)\}$$ is coherent, from which the coherence of $\{\coi(W_x, \iota_x, U_x)\}_{x \in X} \cup \{\coi(W, \iota, U)\}$ follows.  Suppose that $x \in X \cup \omega$, $I \subseteq \pindex(W)$, $I' \subseteq \pindex(W_x)$ and $i \in \{-1, 1\}$ are such that $W \upharpoonright_p I \equiv (W_x \upharpoonright_p I')^i$.  Take $f: I \rightarrow (I')^i$ to be an order isomorphism such that $W \upharpoonright_p \{\lambda\} \equiv (W_x \upharpoonright_p \{f(\lambda)\})^i$.  By the assumptions of (1) we know that there is some $N \in \omega$ such that $I \subseteq I_0\cdots I_N$.  Thus we write

$$
\begin{array}{ll}
[[U \upharpoonright_p \varpropto(I, \iota)]]  & = \prod_{\ell = 0}^N[[U \upharpoonright_p \varpropto(I \cap I_{\ell}, \iota)]] \\
& = \prod_{\ell = 0}^N[[U_{\ell} \upharpoonright_p \varpropto(I \cap I_{\ell}, \iota_{\ell})]]  \\
& = \prod_{\ell \in \{0, \ldots, N\}^i} [[(U_x \upharpoonright_p \varpropto(f(I \cap I_{\ell}) , \iota_x))^i]] \\
& = [[(U_x \upharpoonright_p I')^i]]

\end{array}
$$

\noindent where $\{0, \ldots, N\}^i$ is the set under the reverse order in case $i = -1$.  The first equality is clear, the second is by how $\iota$ is defined, the third is by the coherence of the collection $\{\coi(W_x, \iota_x, U_x)\}_{x \in X} \cup \{\coi(W_n, \iota_n, U_n)\}_{n \in \omega}$, the fourth is evident.  Suppose on the other hand that $I, I' \subseteq \pindex(W)$ and $i \in \{-1, 1\}$ are such that $W \upharpoonright_p I \equiv (W \upharpoonright_p I')^i$.  We consider two cases, assuming without loss of generality that $I \neq \emptyset$ (hence $I' \neq \emptyset$).

\noindent \textbf{Case: $I$ includes into a proper initial subinterval of $\pindex(W)$.}  In this case pick $r, N \in \omega$ for which $I \subseteq I_r \cdots I_N$ and $I \cap I_r \neq \emptyset \neq I \cap I_N$.  As $W \upharpoonright I \in \Pfine(\{W_n\}_{n \in \omega}) \subseteq \Pfine(\{W_x\}_{x \in X})$, by the assumptions of (1) we know $I'$ is also included in a proper initial subinterval of $\pindex(W)$, say $I \subseteq I_{s'} \cdots I_{N'}$ with $I \cap I_{s'} \neq \emptyset \neq I \cap I_{N'}$.  Let $f: I \rightarrow (I')^i$ be an order isomorphism with $W_{\lambda} \equiv (W_{f(\lambda)})^i$.  Let $I = \overline{I_0} \cdots \overline{I_{N''}}$ be the mutual refinement of $I$ with respect to $(I \cap I_s) \cdots (I \cap I_N)$ and $\prod_{r \in \{s', \ldots, N'\}^i} f^{-1}(I' \cap I_r)$.  Note that $\prod_{t \in \{0, N''\}^i} f(\overline{I_t})$ is the mutual refinement of $I'$ with respect to $(I' \cap I_{s'}) \cdots (I' \cap I_{N'})$ and $\prod_{\ell \in \{s, \ldots, N\}^i} f(I \cap I_{\ell})$.  Write $q(t) = \ell$ if $\overline{I_t} \subseteq I \cap I_{\ell}$ and $q'(t) = r$ if $f(I_t) \subseteq I' \cap I_r$.

Then

$$
\begin{array}{ll}
[[U \upharpoonright_p \varpropto(I, \iota)]]  & = \prod_{\ell = s}^N[[U \upharpoonright_p \varpropto(I \cap I_{\ell}, \iota)]] \\
& = \prod_{\ell = s}^N[[U_{\ell} \upharpoonright_p \varpropto(I \cap I_{\ell}, \iota_{\ell})]]  \\
& = \prod_{t = 0}^{N''} [[U_{q(t)} \upharpoonright_p \varpropto(\overline{I_t}, \iota_{q(t)})]]  \\
& = \prod_{t \in \{0, \ldots, N''\}^i} [[(U_{q'(t)} \upharpoonright_p \varpropto(f(\overline{I_t}), \iota_{q'(t)}))^i]]  \\
& = \prod_{r \in \{s', \ldots, N'\}^i} [[(U_r \upharpoonright_p \varpropto(I_r, \iota_r))^i]]\\
& = [[(U \upharpoonright_p \varpropto(I', \iota))^i]].
\end{array}
$$

\noindent \textbf{Case: $I$ does not include into a proper initial subinterval of $\pindex(W)$.}  In this case $I$ is a nonempty terminal interval in $\pindex(W)$.  Then by hypotheses of (1) we know every proper initial p-chunk of $W \upharpoonright_p I$ is an element of $\Pfine(\{W_x\}_{x \in X})$ and every nonempty terminal p-chunk of $W \upharpoonright_p I$ is not an element of $\Pfine(\{W_x\}_{x \in X})$, and this is also the case for $(W\upharpoonright_p I')^i \equiv W \upharpoonright_p I$.  Therefore $i = 1$ and $I'$ is also a nonempty terminal interval in $\pindex(W)$.  We claim $I = I'$ (from which $[[U \upharpoonright_p \varpropto(I, \iota)]] = [[U \upharpoonright_p \varpropto(I', \iota)]]$ is trivial).

To see that $I = I'$ suppose for contradiction, and without loss of generality, that $I$ properly includes $I'$ and $f: I \rightarrow I'$ is an order isomorphism such that $W_{\lambda} \equiv W_{f(\lambda)}$.  Select $\lambda \in I \setminus I'$ and note that $\lambda < f(\lambda) < f(f(\lambda)) < \cdots$, and $W_{\lambda} \equiv W_{f(\lambda)} \equiv W_{f(f(\lambda))} \equiv \cdots$.  Thus the letters in the nonempty word $W_{\lambda}$ are used infinitely often in the word $W$, a contradiction.  This concludes the argument in this case.

One can also check that for $x \in X \cup \omega$, intervals $I \subseteq \pindex(U)$ and $I' \subseteq \pindex(U_x)$, and $i \in \{-1, 1\}$ such that $U \upharpoonright_p I \equiv (U_x \upharpoonright_p I')^i$ we have $[[W \upharpoonright_p \varpropto(I, \iota^{-1})]] = [[(W_x \upharpoonright_p \varpropto(I', \iota_x^{-1}))^i]]$; and also if $I, I' \subseteq \pindex(U)$ and $i \in \{-1, 1\}$ are such that $U \upharpoonright_p I \equiv (U \upharpoonright_p I')^i$ then $[[W \upharpoonright_p \varpropto(I, \iota^{-1})]] = [[(W \upharpoonright_p \varpropto(I', \iota^{-1})^i]]$.  The checks in these cases follow a mutatis mutandis of the arguments which immediately preceded, using the fact that each proper initial p-chunk of $U$ is an element of $\Pfine(\{U_x\}_{x \in X})$ and each nonempty terminal p-chunk of $U$ is not an element of $\Pfine(\{U_x\}_{x \in X})$.  So, we conclude the proof of Lemma \ref{typeomega} (1).
\end{proof}

\begin{proof}[Proof of part (2)]
This claim follows along the same arguments as part (1), except in the choosing of $W$ and $\iota$.  Let $U_n \equiv U \upharpoonright_p I_n$ for each $n \in \omega$.  By Lemma \ref{smallrepresentative} (2) inductively select nonempty $W_n \in \Pfine(\{W_x\}_{x \in X})$ and $\iota_n$, $\dom(\iota_n) \neq \emptyset$, with $d(W_n) > n$ with $\{\coi(W_x, \iota_x, U_x)\}_{x \in X} \cup \{\coi(W_n, \iota_n, U_n)\}_{n = 0}^{\ell}$ coherent for each $\ell \in \omega$.  Now $\{\coi(W_x, \iota_x, U_x)\}_{x \in X} \cup \{\coi(W_n, \iota_n, U_n)\}_{n \in \omega}$ is coherent, and $$\Pfine(\{W_x\}_{x \in X} \cup \{W_n\}_{n \in \omega}) =\Pfine(\{W_x\}_{x \in X})$$ and $$\Pfine(\{U_x\}_{x \in X} \cup \{U_n\}_{n \in \omega}) =\Pfine(\{U_x\}_{x \in X}).$$

To build the word $W$ we select a sequence $V_n$ of words in $\Red_{a, b}$.  Let $V_n'$ be $a_n$, $b_n$, $a_nb_n$, or $b_na_n$ simply so as to have $\pindex(W_nV_n'W_{n + 1}) = \pindex(W_n)\pindex(V_n')\pindex(W_{n + 1})$.  For example, if $W_n$ has a nonempty terminal subword which is $a$-pure and $W_{n + 1}$ has a nonempty terminal initial subword which is $b$-pure then we let $V_n' \equiv b_na_n$.  We will have $V_n \equiv V_n'$ or $V_n \equiv (V_n')^2$.  For either choice of exponent we have $\prod_{n \in \omega} W_nV_n$ is a reduced word and $\pindex(\prod_{n \in \omega} W_nV_n) = \prod_{n \in \omega} \pindex(W_n)\pindex(V_n)$.  If $\Lambda$ is a totally ordered set we let $FI(\Lambda)$ be the collection of nonempty intervals in $\Lambda$ of finite cardinality, with a partial order $\prec$ where $I \prec I'$ if all elements of $I$ are below all elements of $I'$.  Note that for any $V \equiv_p \prod_{\lambda \in \Lambda} V_{\lambda} \in \Red_{a, b}$ and $m \in \omega$ there are only finitely many order monotonic functions $f: \omega \setminus \{0, 1, \ldots, m-1\} \rightarrow FI(\Lambda)$ such that

\begin{itemize}

\item $f(r)$ has cardinality at most $4$;

\item $d(V \upharpoonright_p \{\lambda\}) = r$ for each $\lambda \in f(r)$ and $r \in \omega \setminus \{0, \ldots, m - 1\}$; and

\item $d(V \upharpoonright_p (\max(\min_{\prec}\{f(r), f(r + 1)\}), \min(\max_{\prec}\{f(r), f(r + 1)\})) > r + 1$.

\end{itemize}

\noindent Once $f(m)$ is known, and whether $f$ is increasing or decreasing, we have determined $f$.  Select $\sigma: \omega \rightarrow \{1, 2\}$ such that the word $$W \equiv \prod_{n \in \omega} W_n(V_n')^{\sigma(n)}$$ has no nonempty terminal p-chunk which is a p-chunk of some element in $\{W_x^{\pm 1}\}_{x \in X} \cup \{W_n^{\pm 1}\}_{n \in \omega}$.  Then every nonempty terminal p-chunk of $W$ is not an element of $\Pfine(\{W_x\}_{x \in X})$, but each proper initial p-chunk of $W$ is in $\Pfine(\{W_x\}_{x \in X})$.

Let $\iota = \bigcup_{n \in \omega} \iota_n$ and argue as in part (1) that the collection $\{\coi(W_x, \iota_x, \iota_x)\} \cup \{\coi(W_n, \iota_n, U_n)\}_{n \in \omega} \cup \{\coi(W, \iota, U)\}$ is coherent.
\end{proof}
\end{section}

\begin{section}{$\mathbb{Q}$-type concatenations}\label{nondiscrete}

In this section we prove the following.

\begin{lemma}\label{typeQ}
Suppose that $\{\coi(W_x, \iota_x, U_x)\}_{x \in X}$ is coherent and standardized and that $|X| < 2^{\aleph_0}$.  Each of (1) and (2) below holds.

\begin{enumerate}

\item Suppose further that $W \in \Red_{a, b}$ and $\pindex(W) = \prod_{q \in \mathbb{Q}} I_q$ with

\begin{enumerate}[(i)]

\item each $I_q \neq \emptyset$;

\item $W \upharpoonright_p I_q \in \Pfine(\{W_x\}_{x \in X})$ for each $q \in \mathbb{Q}$; and

\item $W \upharpoonright_p (\bigcup_{\lambda \in \Lambda} I_q) \notin \Pfine(\{W_x\}_{x \in X})$ for each interval $\Lambda \subseteq \mathbb{Q}$ with more than one point. 

\end{enumerate}

\noindent Then there exist $U \in \Red_c$ and coi $\iota$ from $W$ to $U$ such that $$\{\coi(W_x, \iota_x, U_x)\}_{x \in X} \cup \{\coi(W, \iota, U)\}$$ is coherent.

\item Suppose further that $U \in \Red_c$ and $\pindex(U) = \prod_{q \in \mathbb{Q}} I_q$ with

\begin{enumerate}[(i)]

\item each $I_q \neq \emptyset$;

\item $U \upharpoonright_p I_q \in \Pfine(\{U_x\}_{x \in X})$ for each $q \in \mathbb{Q}$; and

\item $U \upharpoonright_p (\bigcup_{q \in \Theta} I_q) \notin \Pfine(\{U_x\}_{x \in X})$ for each interval $\Theta \subseteq \mathbb{Q}$ with more than one point. 

\end{enumerate}

\noindent Then there exist $W \in \Red_{a, b}$ and coi $\iota$ from $W$ to $U$ such that $$\{\coi(W_x, \iota_x, U_x)\}_{x \in X} \cup \{\coi(W, \iota, U)\}$$ is coherent.

\end{enumerate}

\end{lemma}

\begin{proof}[Proof of part (1)]
Assume the hypotheses.  Let $\{W_n\}_{n \in \omega}$ be a list of words such that for each $q \in \mathbb{Q}$ there is some $n \in \omega$ such that $W \upharpoonright_p I_q \equiv W_n$ or $W \upharpoonright_p I_q \equiv W_n^{-1}$ and for $n \neq n'$ we have $W_n \not\equiv W_{n'} \not\equiv W_n^{-1}$.  For simplicity, we require that for each $n \in \omega$ there is indeed some $q \in \mathbb{Q}$ for which $W \upharpoonright_p I_q$ is equal to $W_n$ or $W_n^{-1}$.  In particular this means each $W_n$ is nonempty and an element of $\Pfine(\{W_x\}_{x \in X})$.  As each $I_q$ is nonempty and $\mathbb{Q}$ is infinite, we know that the list $\{W_n\}_{n \in \omega}$ must necessarily be infinite and indeed $\lim_{n \rightarrow \infty} d(W_n) = \infty$.  Since $\mathbb{Q}$ is dense in itself, by condition (iii) we know that if $I \subseteq \pindex(W)$ is an interval such that $W \upharpoonright_p I \in \Pfine(\{W_x\}_{x \in X})$ then $I \subseteq I_q$ for some $q \in \mathbb{Q}$.

Since a nonempty word in $\Red_{a, b}$ is not equal to its inverse (the group $\Red_{a, b}$ is locally free and therefore torsion-free), for each $q \in \mathbb{Q}$ there is a unique choice of $n(q) \in \omega$ and $i(n) \in \{-1, 1\}$ such that $W \upharpoonright_p I_q \equiv W_{n(q)}^{i(q)}$.  Thus we may write $$W \equiv \prod_{q \in \mathbb{Q}} W_{n(q)}^{i(q)}.$$  By Lemma \ref{smallrepresentative} (1) select a word $U_0 \in \Pfine(\{U_x\}_{x \in X})$ and $\iota_0$, with $\dom(\iota_0) \neq \emptyset$ and $d(U_0) > 0$, such that $\{\coi(W_x, \iota_x, U_x)\}_{x \in X} \cup \{\coi(W_0, \iota_0, U_0)\}$ is coherent.   Assuming we have chosen $U_{\ell - 1}$ and $\iota_{\ell - 1}$, choose by Lemma \ref{smallrepresentative} (1) a $U_{\ell} \in \Pfine(\{U_x\}_{x \in X})$ with $d(U_{\ell}) > \ell$ and coi $\iota_{\ell}$ from $W_{\ell}$ to $U_{\ell}$ such that $\{\coi(W_x, \iota_x, U_x)\} \cup \{\coi(W_n, \iota_n, U_n)\}_{n = 0}^{\ell}$ is coherent.  By Lemma \ref{ascendingchaincoi} the collection $\{\coi(W_x, \iota_x, U_x)\}_{x \in X} \cup \{\coi(W_n, \iota_n, U_n)\}_{n \in \omega}$ is coherent.

We shall define a sequence $V_n$ of words in $\Red_c$, where either $V_n \equiv c_n$ or $V_n \equiv c_n^2$.  The word $U$ will be given by the expression $$U \equiv \prod_{q \in \mathbb{Q}} (V_{n(q)}U_{n(q)}V_{n(q)})^{i(q)}.$$  It will require a real argument to show that $U$ is reduced, and the $V_n$ will be chosen in a specific way so that for every interval $I \subseteq \pindex(U)$ such that $U \upharpoonright_p I \in \Pfine(\{U_x\}_{x \in X})$ there exists some $q \in \mathbb{Q}$ for which $I \subseteq \pindex((V_{n(q)}U_{n(q)}V_{n(q)})^{i(q)})$.  We'll first define the $V_n$ and then check why $U$ is reduced.  Note first of all that for each $q \in \mathbb{Q}$ the word $(V_{n(q)}U_{n(q)}V_{n(q)})^{i(q)}$ is reduced, since $d(U_{n(q)}) > n(q)$ and the word $V_{n(q)}$ is $n(q)$-pure.  For the same reason we have in fact that $\pindex((V_{n(q)}U_{n(q)}V_{n(q)})^{i(q)}) = (\pindex(V_{n(q)})\pindex(U_{n(q)})\pindex(V_{n(q)}))^{i(q)}$.  Define a totally ordered set $\Theta$ by $\Theta = \prod_{q \in \mathbb{Q}} \pindex((V_{n(q)}U_{n(q)}V_{n(q)})^{i(q)})$; once we know that $U$ is reduced it will be immediate from the order density of $\mathbb{Q}$ that $\Theta$ is order isomorphic to $\pindex(U)$.

Since the function $n(\cdot): \mathbb{Q} \rightarrow \omega$ is finite-to-one, it is easy (by induction) to find a collection $\{D_s\}_{s \in \omega}$ such that

\begin{itemize}
\item $\mathbb{Q} = \bigsqcup_{s \in \omega} D_s$;

\item each $D_s$ is dense in $\mathbb{Q}$;

\item $n(D_s) \cap n(D_{s'}) = \emptyset$ for $s \neq s'$.
\end{itemize}

\noindent Let $Z: \mathbb{Q} \times \{-1, 1\} \rightarrow \omega$ be a bijection.  For each $q \in \mathbb{Q}$ take $\se^+(q): \omega \rightarrow \mathbb{Q}$ to be a strictly decreasing sequence such that

\begin{itemize}

\item $\lim_{t \rightarrow \infty} \se^+(q)(t) = q$;

\item $(\forall t \in \omega)\se^+(q)(t) \in D_{Z(q, 1)}$;

\item $n(\se^+(q)(t))$ is strictly increasing.
\end{itemize}

\noindent Each $V \in \Red_c$ has the naturally defined function $P_V: \pindex(V) \rightarrow \omega$ where $P_V(\theta) = n$ if $V \upharpoonright_p \{\theta\}$ is $n$-pure.  Let $P: \Theta \rightarrow \omega$ be given by $P(\theta) = n$ if $$\theta \in \pindex((V_{n(q)}U_{n(q)}V_{n(q)})^{i(q)})$$ and $P_{(V_{n(q)}U_{n(q)}V_{n(q)})^{i(q)}}(\theta) = n$.  For each $t \in \omega$ and $q \in \mathbb{Q}$ define $\theta_t \in \Theta$ by $\theta_t = \min(\pindex(\se^+(q)(t)))$.

Fix $q \in \mathbb{Q}$ for the moment.  Note that for each $m \in \omega$ and $V \in \Red_c$ there are at most finitely many p-chunks $V'$ in $V$ for which there is an order isomorphism $L$ from $(\max(\pindex((V_{n(q)}U_{n(q)}V_{n(q)})^{i(q)})), \theta_m] \subseteq \Theta$ to $\pindex(V')$ with $P_V(L(\theta)) = P(\theta)$.  Indeed, once we know $\max\pindex(V')$ we know $V'$ precisely.  Since $|X| < 2^{\aleph_0}$ we can therefore select a sequence $\sigma^+(q): \omega \rightarrow \{1, 2\}$ such that for any 

\begin{itemize}

\item $V \in \{U_x^{\pm 1}\}_{x \in X}$;

\item interval $I \subseteq \pindex(V)$;

\item $m \in \omega$;

\item and order isomorphism $L$ from $(\max(\pindex((V_{n(q)}U_{n(q)}V_{n(q)})^{i(q)})), \theta_m] \subseteq \Theta$ to an interval $\pindex(V)$ with $P_V(L(\theta)) = P(\theta)$

\end{itemize}

\noindent there exists $m < m_0 \in \omega$ such that $$V \upharpoonright_p \{L(\theta_{m_0})\} \not\equiv c_{n(\sigma^+(q)(m_0))}^{\sigma^+(q)(m_0) i(\se^+(q)(m_0))}.$$  For $n \in D_{Z(q, 1)}$ let $V_n \equiv c_n^{\sigma^+(q)(t)}$ if $n = n(\se^+(q)(t))$, and otherwise (i.e. $n \in D_{Z(q, 1)}$ is not in the image of $n(\se^+(q)(\cdot)): \omega \rightarrow \omega$) let $V_n \equiv c_n$.  

Similarly for $q \in \mathbb{Q}$ take $\se^-(q): \omega \rightarrow \mathbb{Q}$ to be a strictly increasing sequence such that

\begin{itemize}

\item $\lim_{t \rightarrow \infty} \se^-(q)(t) = q$;

\item $(\forall t \in \omega)\se^-(q)(t) \in D_{Z(q, -1)}$;

\item $n(\se^-(q)(t))$ is strictly increasing.

\end{itemize}

Define $\theta_t^- \in \Theta$ by $\theta_t^- = \max(\pindex(\se^-(q)(t)))$.  As $|X| < 2^{\aleph_0}$ select a sequence $\sigma^-(q): \omega \rightarrow  \{1, 2\}$ such that for any

\begin{itemize}

\item $V \in \{U_x^{\pm 1}\}_{x \in X}$;

\item interval $I \subseteq \pindex(V)$;

\item $m \in \omega$;

\item and order isomorphism $L$ from $[\theta_m^-, \min(\pindex((V_{n(q)}U_{n(q)}V_{n(q)})^{i(q)}))) \subseteq \Theta$ to some interval $I \subseteq \pindex(V)$ with $P_V(L(\theta)) = P(\theta)$

\end{itemize}

\noindent there exists $m < m_0 \in \omega$ such that $$V \upharpoonright_p \{L(\theta_{m_0}^-)\} \not\equiv c_{n(\sigma^-(q)(m_0))}^{\sigma^-(q)(m_0) i(\se^-(q)(m_0))}.$$  For $n \in D_{Z(q, -1)}$ let $V_n \equiv c_n^{\sigma^-(q)(t)}$ if $n = n(\se^-(q)(t))$, and otherwise (i.e. $n \in D_{Z(q, -1)}$ is not in the image of $n(\se^-(q)(\cdot)): \omega \rightarrow \omega$) let $V_n \equiv c_n$.

Now that we have done this for every $q \in \mathbb{Q}$, we have defined $V_n$ for each $n \in \omega$ and hence have fully defined $U \in \W_c$.  For brevity of notation, write $T_q \equiv (V_{n(q)}U_{n(q)}V_{n(q)})^{i(q)}$, so that $$U \equiv \prod_{q \in \mathbb{Q}} T_q.$$  It is clear that $T_q \in \Red_c$ for each $q \in \mathbb{Q}$.  We verify that $U$ is itself reduced.

\begin{claim}\label{Uisreduced}  $U \in \Red_c$.
\end{claim}

\begin{proof}
Recall that $U$ is a function with domain $\overline{U}$ and $U: \overline{U} \rightarrow \{c_m^{\pm 1}\}_{m \in \omega}$.  We shall suppose for contradiction that $\mathcal{S} \subseteq \overline{U} \times \overline{U}$ is a nonempty cancellation scheme on $U$ (recall Definition \ref{cancellation}).  We shall modify $\mathcal{S}$ into a potentially more intuitive cancellation scheme $\mathcal{S}_{\infty}$, which will pull back to a nonempty cancellation scheme on $W$, giving a contradiction.  If $\overline{\mathcal{S}}$ is a reduction scheme on a word $V$ and $V'$ is a subword of $V$ then write $\pa(V', \overline{\mathcal{S}})$ if for some $r_0 \in \overline{V'}$ there is an $r_1 \in \overline{V}$ for which $\langle r_0, r_1\rangle$ or $\langle r_1, r_0 \rangle$ is in $\overline{\mathcal{S}}$.

Take $N_0 = \min\{n \in \omega \mid (\exists q \in \mathbb{Q}) \pa(T_q, \mathcal{S}) \wedge n = n(q)\}$.  Let $Y = \{q \in \mathbb{Q} \mid n(q) = N_0 \wedge \pa(T_q, \mathcal{S})\}$ and let $Y = \{q_0, \ldots, q_k\}$ list the elements of $Y$ in increasing order.  As $d(U_{N_0}) > N_0$ and $V_{N_0}$ is $N_0$-pure, $T_{q_0}$ is reduced, and by minimaity of $N_0$, we know that there is some $q_j$ with $j > 0$ such that $r_0 = \max\overline{T_{q_0}}$ is paired with one of the first two elements, or one of the last two elements, of $\overline{T_{q_j}}$ in $\mathcal{S}$.  Say $\langle r_0, r_1 \rangle \in \mathcal{S}$, so $i(q_{0, 0}) = -i(q(0, j))$.  Since $U \upharpoonright (r_0, r_1) \sim E$, by counting the number of occurrences of $c_{N_0}$ and of $c_{N_0}^{-1}$, we know that $j$ is odd and $r_1 = \min\overline{T_{q_{0, j}}}$.  Let $\mathcal{S}'$ be the cancellation scheme which includes $\{\langle r, r' \rangle \mid r_0 < r < r_1 \wedge \langle r, r' \rangle \in \mathcal{S}\}$ and which pairs the elements of $\overline{T_{q_{0, 0}}}$ with their inverse counterpart in $\overline{T_{q_{0, j}}}$.

Now we have a nonempty cancellation scheme $\mathcal{S}'$ on $U$ such that if $\pa(T_q, \mathcal{S}')$ then each point in $\overline{T_q}$ appears in an element of $\mathcal{S}'$.  Without loss of generality $\mathcal{S} = \mathcal{S}'$ has this property.  Now we proceed with the argument.  Let $\mathcal{S}_0 = \mathcal{S}$.  Once again define $N_0 = \min\{n \in \omega \mid (\exists q \in \mathbb{Q}) \pa(T_q, \mathcal{S}) \wedge n = n(q)\}$. Let $X_0 = \{q \in \mathbb{Q} \mid n(q) = N_0 \wedge \pa(T_q, \mathcal{S}_0)\}$.  It is easy to see that there is some $q_{0, 0} \in X_0$ whose successor $q_{0, 1}$ in $X_0$ (under the natural order) has $\max \overline{T_{q_{0, 0}}}$ paired up with $\min \overline{T_{q_{0, 1}}}$ in $\mathcal{S}$ (arguing as above).  Then $i(q_{0, 0}) = -i(q_{0, 1})$ and let $f: \overline{T_{q_{0, 0}}} \rightarrow \overline{T_{q_{0, 1}}}$ be order reversing such that $T_{q_{0, 0}}(r)$ is the inverse of $T_{q_{0, 1}}(f(r))$.  Let

$$
\begin{array}{ll}
\mathcal{S}_0^{(1)} & = \{\langle r_0, r_1\rangle \in \mathcal{S}_0 \mid r_0, r_1 \notin \overline{T_{q_{0, 0}}}\cup \overline{T_{q_{0, 1}}}\}\\
& \cup \{\langle r_0, f(r_0) \rangle \mid r_0 \in \overline{T_{q_{0, 0}}}\}\\
& \cup \{\langle r_0, r_1 \rangle \in \overline{U} \times \overline{U} \mid (\exists r_2 \in \overline{T_{q_{0, 0}}}) \langle r_0, r_2\rangle, \langle f(r_2), r_1\rangle \in \mathcal{S}_0\}\\
& \cup \{\langle r_0, r_1 \rangle \in \overline{U} \times \overline{U} \mid (\exists r_2 \in \overline{T_{q_{0, 0}}})\langle r_1, r_2\rangle, \langle r_0, f(r_2) \rangle \in \mathcal{S}_0\}\\
& \cup \{\langle r_0, r_1\rangle \in \overline{U} \times \overline{U} \mid (\exists r_2 \in \overline{T_{q_{0, 0}}})\langle r_2, r_1\rangle, \langle f(r_2), r_0\rangle \in \mathcal{S}_0\}.
\end{array}
$$

Now $\mathcal{S}_0^{(1)}$ is a nonempty cancellation scheme on $U$ for which $\pa(T_q, \mathcal{S}_0^{(1)})$ implies each point in $\overline{T_q}$ appears in an element of $\mathcal{S}_0^{(1)}$, and the elements of $\overline{T_{q_{0, 0}}}$ are paired with those of $\overline{T_{q_{0, 1}}}$ and vice versa.  If $X_0^{(0)} := X_0 \setminus \{q_{0, 0}, q_{0, 1}\}$ is empty then we let $\mathcal{S}_1 = \mathcal{S}_0^{(0)}$, else we find $q_{0, 2}, q_{0, 3} \in X_0^{(1)}$ with  $q_{0, 3}$ the successor of $q_{0, 2}$ in $X_0^{(1)}$ such that $\max\overline{T_{0, 2}}$ is paired with $\min\overline{T_{0, 3}}$ in $\mathcal{S}_0^{(1)}$.  Define as before a new scheme $\mathcal{S}_0^{(2)}$ which pairs the elements of $\overline{T_{0, 2}}$ with those of $\overline{T_{0, 3}}$ and define $X_0^{(2)} = X_0^{(1)} \setminus \{q_{0, 2}, q_{0, 3}\}$.  Continue until $X_0^{(j)} = \emptyset$ and let $\mathcal{S}_1 = \mathcal{S}_0^{(j_0)}$.

Let $N_1 = \min\{n \in \omega \mid (\exists q \in \mathbb{Q}) \pa(T_q, \mathcal{S}_1) \wedge n = n(q) > N_0\}$, and note that the latter set is nonempty (by the order density of $\mathbb{Q}$) so indeed $N_1$ exists.  Let $X_1 = \{q \in \mathbb{Q} \mid n(q) = N_1 \wedge \pa(T_q, \mathcal{S}_1)\}$.  Define $\mathcal{S}_1^{(1)}, \ldots, \mathcal{S}_1^{(j_1)}$ as was done with $\mathcal{S}_0^{(1)}$, etc. to obtain $\mathcal{S}_1^{(j_1)} = \mathcal{S}_2$ so that the elements of $X_1$ are paired up in such a way that 

\begin{center}
$q$ is paired with $q'$ if and only the elements of $\overline{T_q}$ are paired with those in $\overline{T_{q'}}$ under $\mathcal{S}_2$.
\end{center}

\noindent Proceed in these modifications to create $\mathcal{S}_3, \mathcal{S}_4, \ldots$ and sets $X_2, X_3, \ldots$.  Let $\mathcal{S}_{\infty} = \limsup_{\ell \rightarrow \infty} \mathcal{S}_{\ell} = \liminf_{\ell \rightarrow \infty} \mathcal{S}_{\ell} = \bigcup_{j \in \omega} \bigcap_{\ell \geq j} \mathcal{S}_{\ell}$.  Now $\mathcal{S}_{\infty}$ is a nonempty cancellation scheme on $U$ and if $q \in \mathbb{Q}$ has $\pa(T_q, \mathcal{S}_{\infty})$ there is another $q' \in \mathbb{Q}$ with $\pa(T_{q'}, \mathcal{S}_{\infty})$ with $n(q) = n(q')$, $i(q) = -i(q')$, and the elements of $\overline{T_q}$ are paired with the elements of $\overline{T_{q'}}$ under $\mathcal{S}_{\infty}$.  This pairing of elements in $X_{\infty} = \bigcup_{\ell} X_{\ell}$ satisfies properties directly analogous to those of a cancellation scheme, and it is easy to see that this witnesses a nonempty cancellation scheme on $W$ (involving precisely the domains of those subwords $W\upharpoonright I_q$ with $q \in X_{\infty}$), which contradicts the fact that $W$ is reduced.  The proof of Claim \ref{Uisreduced} is complete.
\end{proof}

Now that we know that $U$ is reduced, as was pointed out earlier we may use $\Theta$ for $\pindex(U)$.  By how the $V_n$ were selected, we know that if interval $I \subseteq \Theta$ is such that $U \upharpoonright_p I \in \Pfine(\{U_x\}_{x \in X})$ then there is some $q \in \mathbb{Q}$ for which $I \subseteq \pindex(T_q)$.  The coi $\iota$ from $\pindex(W)$ to $\Theta$ is defined in the most natural way using the sequence of coi $\{\iota_n\}_{n \in \omega}$.  More specifically for each $q \in \mathbb{Q}$ we take $f_q: \pindex(W_{n(q)}) \rightarrow I_q^{i(q)}$ to be an order isomorphism such that  $W_{n(q)} \upharpoonright_p \{\lambda\} \equiv (W \upharpoonright_p \{f_q(\lambda)\})^{i(q)}$.  Similarly let $g_q$ be an order isomorphism with domain $\pindex(U_{n(q)})$ and range $(\pindex(T_q) \setminus \{\min \pindex(T_q), \max \pindex(T_q)\})^{i(q)}$ such that $U_{n(q)} \upharpoonright_p \{\theta\} \equiv (T_q \upharpoonright_p \{g_q(\theta)\} )^{i(q)}$.  Let $\dom(\iota) = \bigcup_{q \in \mathbb{Q}} f_q(\dom(\iota_{n(q)}))$ and $\ran(\iota) = \bigcup_{q \in \mathbb{Q}} g_q(\ran(\iota_{n(q)}))$ and $$\iota(\lambda) = g_q \circ \iota_{n(q)}\circ f_q^{-1}(\lambda).$$

That $\Close(\dom(\iota), \pindex(W))$ and $\Close(\ran(\iota), \Theta)$ are clear by Lemma \ref{basiccloseproperties} (iii), and that $\iota$ is an order isomorphism is clear to see.  It remains to check coherence.

\begin{claim}\label{Qcoi1}  The collection $$\{\coi(W_x, \iota_x, U_x)\}_{x \in X} \cup \{\coi(W_n, \iota_n, U_n)\}_{n \in \omega} \cup \{\coi(W, \iota, U)\}$$ is coherent.
\end{claim}

\begin{proof}
Suppose that $x \in X \cup \omega$, $I \subseteq \pindex(W)$, $I' \subseteq \pindex(W_x)$ and $i \in \{-1, 1\}$ are such that $W \upharpoonright_p I \equiv (W_x \upharpoonright_p I')^i$.  As evidently $W \upharpoonright_p I \in \Pfine(\{W_x\}_{x \in X} \cup \{W_n\}_{n \in \omega}) = \Pfine(\{W_x\}_{x \in X})$, we know by hypothesis (1) (iii) that $I \subseteq I_q$ for some $q \in \mathbb{Q}$.  Therefore

$$
\begin{array}{ll}
[[U \upharpoonright_p \varpropto(I, \iota)]]  & = [[U \upharpoonright_p \varpropto(I, g_q \circ \iota_{n(q)}\circ f_q^{-1})]]\\
& = [[(U_{n(q)} \upharpoonright_p \varpropto(f_q^{-1}(I), \iota_{n(q)}))^{i(q)}]]\\
& = [[(U_x \upharpoonright_p \varpropto(I', \iota_x))^i]]

\end{array}
$$

\noindent where the last equality holds because $\{\coi(W_x, \iota_x, U_x)\}_{x \in X} \cup \{\coi(W_n, \iota_n, U_n)\}_{n \in \omega}$ is coherent and $(W_{n(q)} \upharpoonright_p f_q^{-1}(I))^{i(q)} \equiv W \upharpoonright_p I \equiv (W_x \upharpoonright_p I')^i$.

Suppose now that $I, I' \subseteq \pindex(W)$ and $i \in \{-1, 1\}$ are such that $W \upharpoonright_p I \equiv (W \upharpoonright_p I')^i$.  Let $\Lambda_0, \Lambda_0' \subseteq \mathbb{Q}$ be the intervals given by $\Lambda_0 = \{q \in \mathbb{Q} \mid I_q \cap I \neq \emptyset\}$ and $\Lambda_0' = \{q \in \mathbb{Q} \mid I_q \cap I' \neq \emptyset\}$.  Note the following hold in case $i = 1$ (and respective modifications are given parenthetically for the case $i = -1$):

\begin{itemize}

\item $\Lambda_0$ is empty if and only if $\Lambda_0'$ is empty;

\item $\Lambda_0$ is infinite if and only if $W \upharpoonright_p I \not\in \Pfine(\{W_x\}_{x \in X})$ if and only if $\Lambda_0'$ is infinite;

\item $\Lambda_0$ has a maximal element if and only if $W \upharpoonright_p I$ has a maximal-under-set-inclusion terminal nonempty word which is in $\Pfine(\{W_x\}_{x \in X})$ if and only if $\Lambda_0'$ has a maximal (resp. minimal) element;

\item $\Lambda_0$ has a minimal element if and only if $W \upharpoonright_p I$ has a maximal-under-set-inclusion initial nonempty word which is in $\Pfine(\{W_x\}_{x \in X})$ if and only if $\Lambda_0'$ has a minimal (resp. maximal) element.
\end{itemize}

\noindent Instead of handling a multitude of cases, we'll show that $[[U \upharpoonright_p \varpropto(I, \iota)]] = [[U \upharpoonright_p \varpropto(I', \iota)]]$ in the most elaborate, illustrative case where $i = -1$, $\Lambda_0$ is infinite and has a maximum and no minimum.  Thus $\Lambda_0'$ is infinite and has a minimum and no maximum.  Note the following equivalences

$$
\begin{array}{ll}
W\upharpoonright_p I & \equiv (\prod_{q \in \Lambda_0 \setminus \{\max\Lambda_0\}} W \upharpoonright_p I_q)(W \upharpoonright_p (I_{\max\Lambda_0} \cap I))\\
W \upharpoonright_p I' & \equiv (W \upharpoonright_p (I_{\min\Lambda_0'} \cap I'))(\prod_{q \in \Lambda_0' \setminus \{\min\Lambda_0'\}} W \upharpoonright_p I_q)\\
\prod_{q \in \Lambda_0\setminus \{\max\Lambda_0\}} W \upharpoonright_p I_q & \equiv (\prod_{q \in \Lambda_0' \setminus \{\min\Lambda_0'\}} W \upharpoonright_p I_q)^{-1}\\
W \upharpoonright_p (I_{\max\Lambda_0} \cap I) & \equiv (W \upharpoonright_p (I_{\min\Lambda_0'} \cap I'))^{-1}
\end{array}
$$

\noindent  We have

$$
\begin{array}{ll}
[[U \upharpoonright_p \varpropto(I_{\max\Lambda_0} \cap I, \iota)]] & = [[U_{n(\max\Lambda_0)} \upharpoonright_p \varpropto(f_{\max\Lambda_0}^{-1}(I_{\max\Lambda_0} \cap I), \iota_{n(\max\Lambda_0)})]]\\
& = [[(U_{n(\min\Lambda_0')} \upharpoonright_p \varpropto(f_{\min\Lambda_0'}^{-1}(I_{\min\Lambda_0'} \cap I'), \iota_{n(\min\Lambda_0')}))^{-1}]]\\
& = [[(U \upharpoonright_p \varpropto(I_{\min\Lambda_0'} \cap I', \iota))^{-1}]]
\end{array}
$$

\noindent since $\{\coi(W_n, \iota_n, U_n)\}_{n \in \omega}$ is coherent.  Also since $\Lambda_0 \setminus \{\max\Lambda_0\}$ and $\Lambda_0' \setminus \{\min \Lambda_0'\}$ are each order isomorphic to $\mathbb{Q}$ it is clear that $$\varpropto(\prod_{q \in \Lambda_0 \setminus \{\max\Lambda_0\}}I_q, \iota) = \prod_{q \in \Lambda_0 \setminus\{\max\Lambda_0 \}} \pindex(T_q)$$ and $$\varpropto(\prod_{q \in \Lambda_0 \setminus\{\min\Lambda_0'\}} I_q, \iota) = \prod_{q \in \Lambda_0' \setminus \{\min\Lambda_0'\}}\pindex(T_q).$$  Thus

$$
\begin{array}{ll}
U \upharpoonright_p \varpropto(\prod_{q \in \Lambda_0 \setminus \{\max\Lambda_0\}}I_q, \iota) & \equiv \prod_{q \in \Lambda_0 \setminus \{\max\Lambda_0\}} T_q\\
& \equiv (\prod_{q \in \Lambda_0' \setminus \{\min\Lambda_0'\}} T_q)^{-1}\\
& \equiv (U \upharpoonright_p \varpropto(\prod_{q \in \Lambda_0' \setminus \{\min\Lambda_0'\}} I_q, \iota))^{-1}.
\end{array}
$$

\noindent We obtain

$$
\begin{array}{ll}
[[U \upharpoonright_p I]] & = [[\prod_{q \in \Lambda_0 \setminus \{\max\Lambda_0\}} T_q]][[U \upharpoonright_p \varpropto(I_{\max\Lambda_0} \cap I, \iota)]]\\
& = [[(\prod_{q \in \Lambda_0' \setminus \{\min\Lambda_0'\}} T_q)^{-1}]][[(U \upharpoonright_p \varpropto(I_{\min\Lambda_0'} \cap I', \iota))^{-1}]]\\
& = [[(U \upharpoonright_p I')^{-1}]].
\end{array}
$$

One can also check that for $x \in X \cup \omega$, intervals $I \subseteq \pindex(U)$ and $I' \subseteq \pindex(U_x)$, and $i \in \{-1, 1\}$ such that $U \upharpoonright_p I \equiv (U_x \upharpoonright_p I')^i$ we have $[[W \upharpoonright_p \varpropto(I, \iota^{-1})]] = [[(W_x \upharpoonright_p \varpropto(I', \iota_x^{-1}))^i]]$; and also if $I, I' \subseteq \pindex(U)$ and $i \in \{-1, 1\}$ are such that $U \upharpoonright_p I \equiv (U \upharpoonright_p I')^i$ then $[[W \upharpoonright_p \varpropto(I, \iota^{-1})]] = [[(W \upharpoonright_p \varpropto(I', \iota^{-1})^i]]$.  This is done using the same strategy, using the important fact that if interval $I \subseteq \pindex(U)$ is such that $U \upharpoonright_p I \in \Pfine(\{U_x\}_{x \in X})$ then there is some $q \in \mathbb{Q}$ for which $I \subseteq \pindex(T_q)$.  Thus Claim \ref{Qcoi1} is proved.
\end{proof}

The proof of part (1) is now finished.
\end{proof}

\begin{proof}[Proof of part (2)]
The proof of part (2) is quite similar and we will spare the details when the arguments are nearly identical.  Take $U$ and a decomposition $\pindex(U) = \prod_{q \in \mathbb{Q}} I_q$ as in the hypotheses.  Let $\{U_n\}_{n \in \omega}$ be a list of words in $\Pfine(\{U_x\}_x \in X)$ such that

\begin{itemize}

\item for each $q \in \mathbb{Q}$ there is $n \in \omega$ with $U \upharpoonright_p I_q \equiv U_n$ or $U \upharpoonright_p I_q \equiv U_n^{-1}$;

\item for $n \neq n'$ we have $U_n \not\equiv U_{n'} \not\equiv U_n^{-1}$;

\item for each $n \in \omega$ there exists $q \in \mathbb{Q}$ with $U \upharpoonright_p I_q \equiv U_n$ or $U \upharpoonright_p I_q \equiv U_n^{-1}$. 

\end{itemize}

\noindent Define functions $n(\cdot): \mathbb{Q} \rightarrow \omega$ and $i(\cdot): \mathbb{Q} \rightarrow \{-1, 1\}$ by $U \upharpoonright_p I_q \equiv U_{n(q)}^{i(q)}$.  Pick by induction (Lemma \ref{smallrepresentative} (2)) a sequence $\{\coi(W_n, \iota_n, U_n)\}_{n \in \omega}$ of coi triples such that

\begin{itemize}

\item $W_n \in \Pfine(\{W_x\}_{x \in X})$;

\item $d(W_n) > n$;

\item $\iota_n$ has nonempty domain;

\item $\{\coi(W_x, \iota_x, U_x)\}_{x \in X} \cup \{\coi(W_n, \iota_n, U_n)\}_{n \in \omega}$ is coherent.

\end{itemize}

For each $n \in \omega$ we select words $V_{n, 0}, V_{n, 1}, V_n \in \Red_{a, b}$.  Let $V_{n, 0} \equiv b_n$ if $W_n$ has a nonempty initial subword which is $a$-pure, otherwise let $V_{n, 0} \equiv E$.  Let $V_{n, 1} \equiv b_n$ if $W_n$ has a nonempty terminal subword which is $a$-pure, otherwise let $V_{n, 1} \equiv E$.  The word $V_n$ will either be $a_n$ or $a_n^2$, chosen to ensure that the word $W \equiv \prod_{q \in \mathbb{Q}} (V_{n(q)}V_{n(q), 0}W_{n(q)}V_{n(q), 1}V_{n(q)})^{i(q)}$ does not have large p-chunks which are in $\Pfine(\{W_x\}_{x \in X})$.  The words $V_{n, 0}$ and $V_{n, 1}$ simply serve as a `buffer' so that the reduced word $V_nV_{n, 0}W_nV_{n, 1}V_n$ will have the property $$\pindex(V_nV_{n, 0}W_nV_{n, 1}V_n) = \pindex(V_n)\pindex(V_{n, 0})\pindex(W_n)\pindex(V_{n, 1})\pindex(V_n).$$  We have defined the words $V_{n, 0}$ and $V_{n, 1}$, and it remains to define $V_n$.  Define totally ordered set $\Lambda$ by $$\Lambda = \prod_{q \in \mathbb{Q}}\pindex((V_{n(q)}V_{n(q), 0}W_{n(q)}V_{n(q), 1}V_{n(q)})^{i(q)}).$$  The selection of $V_n$ differs slightly from how it was done in the proof of (1), particularly the functions $P_V$ and $P$ are a little different.

As in part (1) construct a collection $\{D_s\}_{s \in \omega}$ such that

\begin{itemize}
\item $\mathbb{Q} = \bigsqcup_{s \in \omega} D_s$;

\item each $D_s$ is dense in $\mathbb{Q}$;

\item $n(D_s) \cap n(D_{s'}) = \emptyset$ for $s \neq s'$
\end{itemize}

\noindent and let $Z: \mathbb{Q} \times \{-1, 1\} \rightarrow \omega$ be a bijection.  For each $q \in \mathbb{Q}$ pick $\se^+(q): \omega \rightarrow \mathbb{Q}$ to be a strictly decreasing sequence such that

\begin{itemize}

\item $\lim_{t \rightarrow \infty} \se^+(q)(t) = q$;

\item $(\forall t \in \omega)\se^+(q)(t) \in D_{Z(q, 1)}$;

\item $n(\se^+(q)(t))$ is strictly increasing.
\end{itemize}

\noindent Note that if $V \in \Red_{a, b}$ we have a function $P_V: \pindex(V) \rightarrow \omega$ given by $P_v(\lambda) = d(V \upharpoonright_p \{\lambda\})$.  We point out that although this function looks formally different from its counterpart in the proof of part (1), its definition actually yields the same values for elements in $\Red_c$ since for $z \in \mathbb{Z} \setminus \{0\}$ we have $d(c_n^z) = n$ is $n$-pure.  Let $P: \Lambda \rightarrow \omega$ be given by $P(\lambda) = n$ if $$\lambda \in \pindex((V_{n(q)}V_{n(q), 0}W_{n(q)}V_{n(q), 1}V_{n(q)})^{i(q)})$$ and $$P_{(V_{n(q)}V_{n(q), 0}W_{n(q)}V_{n(q), 1}V_{n(q)})^{i(q)}}(\lambda) = n.$$

Fix $q \in \mathbb{Q}$.  Define $\lambda_t \in \Lambda$ by $\lambda_t = \min(\pindex(\se^+(q)))$, and from the fact that $|X| < 2^{\aleph_0}$ we select $\sigma^+(q): \omega \rightarrow \{1, 2\}$ such that for any

\begin{itemize}

\item $V \in \{W_x^{\pm 1}\}_{x \in X}$;

\item interval $I \subseteq \pindex(V)$;

\item $m \in \omega$;

\item and order isomorphism $L$ from $$(\max(\pindex((V_{n(q)}V_{n(q), 0}U_{n(q)}V_{n(q), 1}V_{n(q)})^{i(q)})), \lambda_m] \subseteq \Lambda$$ to some interval $I \subseteq \pindex(V)$ with $P_V(L(\lambda)) = P(\lambda)$

\end{itemize}

\noindent there exists $m < m_0 \in \omega$ such that $$V \upharpoonright_p \{L(\lambda_{m_0}^-)\} \not\equiv a_{n(\sigma^+(q)(m_0))}^{\sigma^+(q)(m_0) i(\se^+(q)(m_0))}.$$  For $n \in D_{Z(q, 1)}$ let $V_n \equiv a_n^{\sigma^+(q)(t)}$ if $n = n(\se^+(q)(t))$ and otherwise let $V_n \equiv a_n$.

Define for each $q \in \mathbb{Q}$ the analogous $\se^-(Q)$ as in part (1), together with $\sigma^-(q) : \omega \rightarrow \{1, 2\}$ satisfying the comparable qualities.  For $n \in D_{Z(q, -1)}$ let $V_n \equiv a_n^{\sigma^-(q)(t)}$ if $n = n(\se^-(q)(t))$ and otherwise $V_n \equiv a_n$.  We have $W \in \Red_{a, b}$ by arguing as in part (1) with the obvious modifications.  Also, by how the $\sigma^+$ and $\sigma^-$ were defined, we know that if $I \subseteq \pindex(W)$ is such that $W \upharpoonright_p I 
\in \Pfine(\{W_x\}_{x \in X})$ then $I \subseteq \pindex((V_{n(q)}V_{n(q), 0}U_{n(q)}V_{n(q), 1}V_{n(q)})^{i(q)})$ for some $q \in \mathbb{Q}$.  Define a coi $\iota$ from $W$ to $U$ precisely as in part (1), and the check that $$\{\coi(W_x, \iota_x, U_x)\}_{x \in X} \cup \{\coi(W_n, \iota_n, U_n)\}_{n \in \omega} \cup \{\coi(W, \iota, U)\}$$ is coherent is the same.  This completes the proof of Lemma \ref{typeQ} (2).
\end{proof}
\end{section}

\begin{section}{The proof of Theorem \ref{themain}}\label{Concludingarg}

We prove the following lemma and then give the proof of Theorem \ref{themain}.

\begin{lemma}\label{arbextension}
Suppose that $\{\coi(W_x, \iota_x, U_x)\}_{x \in X}$ is coherent and standardized and that $|X| < 2^{\aleph_0}$.  Each of the following holds.

\begin{enumerate}
\item  If $W \in \Red_{a, b}$ then there exists $U \in \Red_c$ and coi $\iota$ from $W$ to $U$ such that $\{\coi(W_x, \iota_x, U_x)\}_{x \in X}$ is coherent.

\item If $U \in \Red_c$ then there exists $W \in \Red_{a, b}$ and coi $\iota$ from $W$ to $U$ such that $\{\coi(W_x, \iota_x, U_x)\}_{x \in X}$ is coherent.

\end{enumerate}
\end{lemma}

\begin{proof}
(1)  Assume the hypotheses.  Let $\Lambda = \pindex(W)$.  If $\Lambda = \emptyset$ then $W \equiv E$ and we let $U \equiv E$ and $\iota$ be the empty function.  If $\Lambda \neq \emptyset$ then for each $\lambda \in \Lambda$ we add the coi triple $\coi(W \upharpoonright_p \{\lambda\}, \iota_{\lambda}, E)$ to the collection $\{\coi(W_x, \iota_x, U_x)\}_{x \in X}$.  Letting $X'$ denote the new index for this enlarged collection it is easy to see that the new collection $\{\coi(W_x, \iota_x, U_x)\}_{x \in X'}$ is coherent.  If $C$ is a coherent collection of coi triples we let $h(C)$ denote the set of words that appear in the first coordinate of the elements of $C$.  Let $C_0 = \{\coi(W_x, \iota_x, U_x)\}_{x \in X'}$, so $h(C_0) = \{W_x\}_{x \in X'}$.  Note that each $\lambda \in \Lambda$ is contained in an interval $I$, namely $I = \{\lambda\}$, of $\Lambda$ such that $W \upharpoonright_p I \in \Pfine(h(C_0))$.  Also, $|C_0| < 2^{\aleph_0}$.  Using the lemmas that have been proved so far we will extend to coherent collections $C_0 \subseteq C_1 \subseteq C_2, \ldots, C_{\alpha}$ where $\alpha$ is a countable ordinal.  Let $\omega_1$ denote the collection of countable ordinals.  Our construction will consist in two major steps.

\noindent\textbf{Step 1.}

Suppose that we have defined coherent coi collections $C_{\gamma}$ for all $\gamma < \alpha \in \omega_1$ to be nesting, $\subseteq$-increasing and such that $|C_{\gamma}| < 2^{\aleph_0}$.  If $\alpha > 0$ is a limit ordinal then let $C_{\alpha} = \bigcup_{\gamma < \alpha} C_{\alpha}$, and note that $C_{\alpha}$ is coherent by Lemma \ref{ascendingchaincoi}, and $|C_{\alpha}| < 2^{\aleph_0}$ since $2^{\aleph_0}$ is not of countable cofinality \cite[Theorem 3.11]{Jech}.  If each $\lambda \in \Lambda$ is contained in a maximal interval $I_{\lambda} \subseteq \Lambda$ for which $W \upharpoonright_p I_{\lambda} \in \Pfine(h(C_{\alpha}))$ then proceed to Step 2 (by maximal we mean that there is no strictly larger interval $\Lambda \supseteq I \supseteq I_{\lambda}$ for which $W \upharpoonright_p I \in \Pfine(h(C_0))$).  Else we execute the following procedure.

\noindent \textbf{Case (i).}  There exists $\lambda \in \Lambda$ and a sequence of intervals $\{I_{\ell}\}_{\ell \in \omega}$ such that $I_{\ell}$ is a strict subset of $I_{\ell + 1}$ such that

\begin{itemize}

\item $\lambda = \min I_{\ell}$;

\item $W \upharpoonright_p I_{\ell} \in \Pfine(h(C_{\alpha}))$;

\item $W \upharpoonright_p \bigcup_{\ell \in \omega} I_{\ell} \notin \Pfine(h(C_{\alpha}))$.

\end{itemize}

\noindent In this case select such a $\lambda$ and sequence $\{I_{\ell}\}_{\ell \in \omega}$ and write $W_{\alpha} \equiv W \upharpoonright_p \bigcup_{\ell \in \omega}I_{\ell}$.  By Lemma \ref{typeomega} (1) select $U_{\alpha} \in \Red_c$ and coi $\iota$ from $W_{\alpha}$ to $U_{\alpha}$ such that $C_{\alpha + 1} = C_{\alpha} \cup \{\coi(W_{\alpha}, \iota_{\alpha}, U_{\alpha})\}$ is coherent.

\noindent \textbf{Case (ii).}  If Case (i) fails then there exists $\lambda \in \Lambda$ and a sequence of intervals $\{I_{\ell}\}_{\ell \in \omega}$ such that $I_{\ell}$ is a strict subset of $I_{\ell + 1}$ such that

\begin{itemize}

\item $\lambda = \max I_{\ell}$;

\item $W \upharpoonright_p I_{\ell} \in \Pfine(h(C_{\alpha}))$;

\item $W \upharpoonright_p \bigcup_{\ell \in \omega} I_{\ell} \notin \Pfine(h(C_{\alpha}))$.

\end{itemize}

\noindent In this case we select such a $\lambda$ and sequence $\{I_{\ell}\}_{\ell \in \omega}$ and write $W_{\alpha} \equiv W \upharpoonright_p \bigcup_{\ell \in \omega}I_{\ell}$.  By Lemma \ref{typeomega} (1) (applied to $W_{\alpha}^{-1}$) select $U_{\alpha} \in \Red_c$ and coi $\iota$ from $W_{\alpha}$ to $U_{\alpha}$ such that $C_{\alpha + 1} = C_{\alpha} \cup \{\coi(W_{\alpha}, \iota_{\alpha}, U_{\alpha})\}$ is coherent.

We claim that for some $\alpha \in \omega_1$ this procedure has terminated and we have moved on to Step (2) (in other words, each $\lambda \in \Lambda$ is contained in a maximal interval $I_{\lambda} \subseteq \Lambda$ for which $W \upharpoonright_p I_{\lambda} \in \Pfine(h(C_{\alpha}))$).  Indeed, if this is not the case we define a function $f: \omega_1 \rightarrow \Lambda \times \{(i), (ii)\}$ by letting $f(\alpha) = (\lambda, (i))$ if at stage $\alpha + 1$ we were in Case (i) and used $\lambda$ as the minimal point of the intervals (and $f(\alpha) = (\lambda, (ii))$ if at stage $\alpha + 1$ we were in Case (ii) and used $\lambda$ as the maximal point of the intervals).  Since $\omega_1$ is uncountable and the set $\Lambda \times \{(i), (ii)\}$ is countable, there is an uncountable subset $Y \subseteq \omega_1$ on which the restriction $f \upharpoonright Y$ is constant.

If, say, $f$ is constantly $(\lambda, (i))$ on $Y$ then we let $I_{\alpha}$ be the unique interval in $\Lambda$ with $\lambda = \min I_{\alpha}$ and $W \upharpoonright_p I_{\alpha} \equiv W_{\alpha}$.  By construction we have $I_{\alpha}$ strictly includes into $I_{\alpha'}$ for elements $\alpha < \alpha'$ in $Y$.  Select $\lambda_{\alpha} \in I_{\min\{\alpha' \in Y \mid \alpha' > \alpha\}} \setminus I_{\alpha}$ and note that this gives an injection $\alpha \mapsto \lambda_{\alpha}$ from the uncountable set $Y$ to the countable set $\Lambda$, a contradiction.  In case $f$ is constantly $(\lambda, (ii))$ on $Y$ then the argument is similar.  Thus we know that Step (1) has terminated at, say, $C_{\alpha}$ where $\alpha \in \omega_1$.

\noindent\textbf{Step 2.}

Now we have a standardized coherent collection $C_{\alpha} \supseteq \{\coi(W_x, \iota_x, U_x)\}_{x \in X}$ such that $|C_{\alpha}| < 2^{\aleph_0}$ and each $\lambda \in \Lambda$ is contained in a maximal interval $I_{\lambda} \subseteq \Lambda$ such that $W \upharpoonright_p I_{\lambda} \in \Pfine(C_{\alpha})$.  Taking $\mathcal{I} = \{I_{\lambda}\}_{\lambda \in \Lambda}$ to be this collection of such maximal intervals it is clear that if $I_{\lambda_0} \cap I_{\lambda_1} \neq \emptyset$ then $I_{\lambda_0} = I_{\lambda_1}$, so the collection $\mathcal{I}$ gives a decomposition of $\Lambda$ into nonempty pairwise disjoint intervals.  The linear order on $\Lambda$ descends naturally to a linear order $\prec$ on $\mathcal{I}$ via comparison of elements.

If $\mathcal{I}$ has exactly one element $\mathcal{I} = \{I\}$ then $I = \Lambda$ and so $W = W \upharpoonright_p I \in \Pfine(h(C_{\alpha}))$.  Thus by Lemma \ref{findsomerepresentative} (1) we find $U \in \Red_c$ and coi $\iota$ from $W$ to $U$ such that $C_{\alpha} \cup \{\coi(W, \iota, U)\}$ is coherent, so $\{\coi(W_x, \iota_x, U_x)\}_{x \in X} \cup \{\coi(W, \iota, U)\}$ is coherent.  If instead $\mathcal{I}$ has at least two elements, then we point out that the order $\prec$ on $\mathcal{I}$ is dense (for, if $I'$ is the immediate successor of $I$ under $\prec$ then $W \upharpoonright_p (I \cup I') \equiv (W \upharpoonright_p I)(W \upharpoonright_p I') \in \Pfine(h(C_{\alpha}))$ contrary to the maximality of $I$ and $I'$).  Let $\mathcal{I}'$ be the set $\mathcal{I}$ minus its maximal and/or minimal element (if such exist).  Now $\mathcal{I}'$ is order isomorphic to $\mathbb{Q}$, and we apply Lemma \ref{typeQ} (1) to obtain a word $U' \in \Red_c$ and coi $\iota'$ from $W \upharpoonright_p \bigcup \mathcal{I}'$ to $U'$ so that the collection $C_{\alpha} \cup \{\coi(W \upharpoonright_p \bigcup \mathcal{I}', \iota', U' )\}$ is coherent.  Now $$W \in \Pfine(h(C_{\alpha} \cup \{\coi(W \upharpoonright_p \bigcup \mathcal{I}', \iota', U' )\}))$$ since $W$ is a concatention of at most three elements in $\Pfine(h(C_{\alpha} \cup \{\coi(W \upharpoonright_p \bigcup \mathcal{I}', \iota', U' )\}))$.  Apply Lemma \ref{findsomerepresentative} (1) to find $U \in \Red_c$ and coi $\iota$ from $W$ to $U$ such that $$C_{\alpha} \cup \{\coi(W \upharpoonright_p \bigcup \mathcal{I}', \iota', U' )\} \cup \{\coi(W, \iota, U)\}$$ is coherent.  Thus $\{\coi(W_x, \iota_x, U_x)\} \cup \{\coi(W, \iota, U)\}$ is coherent.  This completes the proof of (1).

The proof of (2) follows the same steps under the obvious changes.
\end{proof}

\begin{proof}[Proof of Theorem \ref{themain}]
We know that $|\Red_{a, b}| = |\Red_c| = 2^{\aleph_0}$.  Let $\mathfrak{c}$ denote the smallest ordinal of cardinality $2^{\aleph_0}$.  As is well-known, each ordinal $\alpha$ can be expressed uniquely as a sum $\alpha = \gamma + n$ where $\gamma$ is a limit ordinal (possibly $\gamma = 0$) and $n$ is a natural number.  So, an ordinal may be said to be even or odd depending on the natural number $n$.  Let $\mathbb{E} = \{\alpha < \mathfrak{c} \mid \alpha \text{ is even}\}$ and $\mathbb{O} = \mathfrak{c} \setminus \mathbb{E}$.  Let $\Red_c = \{U_{\alpha}\}_{\alpha \in \mathbb{O}}$ be an enumeration such that for $n \in \mathbb{O} \cap \omega$ we have $U_n \equiv c_{\frac{n - 1}{2}}$.  Let $\Red_{a, b} = \{W_{\alpha}\}_{\alpha \in \mathbb{E}}$ be an enumeration such that for $n \in \mathbb{E} \cap \omega$ we have

\[
W_n = \left\{
\begin{array}{ll}
a_{\frac{n}{4}}
                                            & \text{if } n \in 4\omega, \\
b_{\frac{n - 2}{4}}                                        & \text{if } n \in 2\omega \setminus 4\omega.
\end{array}
\right.
\]

\noindent For $n \in \mathbb{O} \cap \omega$ let $W_n \equiv E$ and $\iota_n$ be the empty function, and similarly for $n \in \mathbb{E} \cap \omega$ let $U_n \equiv E$ and $\iota_n$ be the empty function.  Now the collection $\{\coi(W_n, \iota_n, U_n)\}_{n \in \omega}$ is coherent and standardized.

Now we continue the construction of our coherent collection by induction.  Suppose $\omega \leq \alpha < \mathfrak{c}$ and $\{\coi(W_{\beta}, \iota_{\beta}, U_{\beta})\}_{\beta < \alpha}$ is a coherent standardized collection.  If $\alpha \in \mathbb{E}$ then by Lemma \ref{arbextension} (1) select $U_{\alpha} \in \Red_c$ and coi $\iota_{\alpha}$ from $W_{\alpha}$ to $U_{\alpha}$ such that the collection $\{\coi(W_{\beta}, \iota_{\beta}, U_{\beta})\}_{\beta \leq \alpha}$ is coherent.  If $\alpha \in \mathbb{O}$ then by Lemma \ref{arbextension} (2) select $W_{\alpha} \in \Red_{a, b}$ and coi $\iota_{\alpha}$ from $W_{\alpha}$ to $U_{\alpha}$ such that the collection $\{\coi(W_{\beta}, \iota_{\beta}, U_{\beta})\}_{\beta \leq \alpha}$ is coherent.  In the end we have constructed a coherent collection $\{\coi(W_{\beta}, \iota_{\beta}, U_{\beta})\}_{\beta < \mathfrak{c}}$ such that $\Pfine(\{W_{\beta}\}_{\beta < \mathfrak{c}}) = \Red_{a, b}$ and $\Pfine(\{U_{\beta}\}_{\beta < \mathfrak{c}}) = \Red_c$.  By Lemma \ref{homomorphism} we see that $\Red_{a, b}/\langle\langle \Pure_{a, b}\rangle\rangle$ is isomorphic to $\Red_c/\langle\langle \Pure_c \rangle\rangle$.  Theorem \ref{themain} is immediate from Lemmas \ref{GScombinatorial} and \ref{HAcombinatorial}.

\end{proof}

\end{section}

\section*{Acknowledgement}

The author thanks Jeremy Brazas for the beautiful pictures used in this article.

\end{document}